\documentclass[reqno,12pt]{amsart}
\usepackage{amssymb}
\usepackage{amsmath}
\usepackage{mathrsfs}
\usepackage{amsmath,amssymb}
\usepackage{paralist}
\usepackage{graphics} 
\usepackage{epsfig} 
\usepackage[colorlinks = true]{hyperref}
\hypersetup{urlcolor = red, citecolor = blue}
\usepackage[top = 1in,bottom = 1in,left = 1in,right = 1in]{geometry}
\usepackage{cite}
\newtheorem{thm}{Theorem}[section]

\newtheorem{lem}[thm]{Lemma}
\newtheorem{prop}[thm]{Proposition}
\theoremstyle{definition}

\newtheorem{rem}{Remark}[section]
\numberwithin{equation}{section}

\begin{document}

\title[ repulsion-consumption Keller-Segel system]
{Boundedness and finite-time blow-up in a repulsion-consumption system with nonlinear chemotactic sensitivity}

\author[Zeng]{Ziyue Zeng}%
\address{School of Mathematics, Southeast University, Nanjing 211189, P. R. China}
\email{ziyzzy@163.com}

\author[Li]{Li Yuxiang$^{\star}$}

\thanks{$^{\star}$Corresponding author}
\address{School of Mathematics, Southeast University, Nanjing 211189,  P. R. China}
\email{lieyx@seu.edu.cn}

\thanks{Supported in part by National Natural Science Foundation of China (No. 12271092, No. 11671079) and the Jiangsu Provincial Scientific Research Center of Applied Mathematics (No. BK20233002).}



\maketitle

\begin{abstract}
This paper investigates the repulsion-consumption system
\begin{align}\tag{$\star$}
  \left\{
  \begin{array}{ll}
    u_t=\Delta u+\nabla \cdot(S(u) \nabla v),  \\
    \tau v_t=\Delta v-u v, 
\end{array}
  \right.
\end{align}
under no-flux/Dirichlet conditions for $u$ and $v$ in a ball $B_R(0) \subset \mathbb R^n $. When $\tau=\{0,1\}$ and $0<S(u)\leqslant K(1+u)^{\beta}$ for $u \geqslant 0$ with some $\beta \in (0,\frac{n+2}{2n})$ and $K>0$, we show that for any given radially symmetric initial data, the problem ($\star$) possesses a global bounded classical solution. Conversely, when $\tau=0$, $n=2$ and $S(u) \geqslant k u^{\beta}$ for $u \geqslant 0$ with some $\beta>1$ and $k>0$, for any given initial data $u_0$, there exists a constant $M^{\star}=M^{\star}\left(u_0\right)>0$ with the property that whenever the boundary signal level $M\geqslant M^{\star}$, the corresponding radially symmetric solution blows up in finite time. 

Our results can be compared with that of the papers [J.~Ahn and M.~Winkler, {\it Calc. Var.} {\bf 64} (2023).] and [Y. Wang and M. Winkler, {\it Proc. Roy. Soc. Edinburgh Sect. A}, \textbf{153} (2023).], in which the authors studied the system ($\star$) with the first equation replaced respectively by $u_t=\nabla \cdot ((1+u)^{-\alpha} \nabla u)+\nabla \cdot(u \nabla v)$ and $u_t=\nabla \cdot ((1+u)^{-\alpha} \nabla u)+\nabla \cdot(\frac{u}{v} \nabla v)$. Among other things, they obtained that, under some conditions on $u_0(x)$ and the boundary signal level, there exists a classical solution blowing up in finite time whenever $\alpha>0$.
\end{abstract}

\section{Introduction}
\vskip 3mm

In this paper, we consider the following repulsion-consumption system
\begin{align}\label{1.0.0}
  \left\{
  \begin{array}{ll}
    u_t=\Delta u+\nabla \cdot\left(S(u) \nabla v\right), & x \in \Omega, t>0, \\
    \tau v_t=\Delta v-u v, & x \in \Omega, t>0, \\
    \left( \nabla u+{S(u)} \nabla v\right) \cdot \nu=0, \quad v=M, & x \in \partial \Omega, t>0, \\
    u(x, 0)=u_0(x), \quad \tau v(x, 0)=\tau v_0(x), & x \in \Omega,
  \end{array}
  \right.
\end{align}
where $\Omega=B_R(0) \subset \mathbb {R}^n$ is a ball, $M > 0$ is a given parameter and $\tau=\{0,1\}$. The scalar functions $u$ and $v$ denote the cell density and the chemical concentration consumed by cells, respectively. $S(u)$ represents the chemotactic sensitivity of cells, which generalizes the prototypical choice in $S(u)=u^{\beta}$ for $u \geqslant 0 $ with some $\beta>0$. The main purpose of this paper is to determine the explosion-critical parameter associated with the chemotactic sensitivity function $S(u)$ in system (\ref{1.0.0}).

The system (\ref{1.0.0}) originates from the chemotaxis-consumption system, which describes the intricate patterns formed by the colonies of Bacillus subtilis as they seek oxygen  \cite{1990-PA-MATSUSHITAFUJIKAWA , 2005-PTNASTUSA-TuvalCisnerosDombrowskiWolgemuthKesslerGoldstein , 1971-JOTB-KELLERSEGEL},
\begin{align}\label{1.0.9}
  \left\{
  \begin{array}{ll}
    u_t=\nabla \cdot(D(u) \nabla u)-\nabla \cdot(uS(u,v) \nabla v), \\
    v_t=\Delta v-u v,  
 \end{array}
  \right.
\end{align} 
 where $D(u)$ and $S(u,v)$ denote the diffusivity and the chemotactic sensitivity of cells, respectively. The system (\ref{1.0.9}) with $D(u)= 1$ and $S(u,v)= \chi $, subjected to homogeneous Neumann boundary conditions, has been investigated extensively. If $n\geqslant 2$, Tao \cite{ 2011-JMAA-Tao } obtained a global bounded classical solution for system (\ref{1.0.9}) under the condition that $\left\|v_0\right\|_{L^{\infty}(\Omega)}$ is sufficiently small. Zhang and Li \cite{2015-JMP-ZhangLi} demonstrated that if either $n \leqslant 2$ or $n \geqslant 3$ and $0<\chi \leqslant \frac{1}{6(n+1)\|v(x, 0)\|_{L^{\infty}(\Omega)}}$, the global classical solution $(u,v)$ converges to $(\frac{1}{|\Omega|} \int_{\Omega} u_0,0)$ exponentially as $t \rightarrow \infty$. For arbitrary large initial data, Tao and Winkler \cite{2012-JDE-TaoWinklera} proved that when $n=2$, the global classical solution of (\ref{1.0.9}) is bounded, satisfying $(u,v) \rightarrow (\frac{1}{|\Omega|} \int_{\Omega} u_0,0)$  as $t \rightarrow \infty$; when $n=3$, the problem admits global weak solutions, which eventually become bounded and smooth. Moreover, such solutions also approach the spatially constant equilibria $(\frac{1}{|\Omega|} \int_{\Omega} u_0,0)$ in the large time limit. When $n \geqslant 4$, Wang and Li \cite{2019-EJDE-WangLi} showed that this model possesses at least one global renormalized solution. For more details about the modeling of chemotaxis-consumption models, we refer to the survey \cite{2023-SAM-LankeitWinkler}.
 
When $D(u)$ extends the prototypical choice in $D(u)= c_D u^{m-1}$ for $u>0$, there are also some results for system (\ref{1.0.9}) subjected to homogeneous Neumann boundary conditions. Among the results obtained by Wang et al. \cite{2014-ZAMP-WangMuZhou}, if $m>2-\frac{2}{n}$ with $n \geq 2$, the system (\ref{1.0.9}) with a sufficiently smooth $S(v)$ possesses a unique global bounded classical solution in a convex smooth bounded domain. In the case that $S(u,v) = 1$ and $n \geqslant 3$, some authors conducted further researches. Relaxing the assumption that $m>2-\frac{2}{n}$ in \cite{2014-ZAMP-WangMuZhou} to $m>2-\frac{6}{n+4}$, in a convex domain, Wang et al. \cite{2015-ZAMP-WangMuLinZhao} obtained global classical solutions for system (\ref{1.0.9}) in the case of non-degenerate diffusion, and global weak solutions for system (\ref{1.0.9}) in the degenerate case. Subsequently, for the degenerate diffusion, Wang and Xiang \cite{2015-ZAMP-WangXiang} removed the convexity assumption on the domain in \cite{2014-ZAMP-WangMuZhou, 2015-ZAMP-WangMuLinZhao} and established global bounded weak solutions
under the condition that $m>2-\frac{n+2}{2 n}$. In general domain, the range of $m$ was relaxed from  $m>2-\frac{2}{n}$ in \cite{2014-ZAMP-WangMuZhou} to $m>\frac{3}{2}-\frac{1}{n}$ by Fan and Jin\cite{2017-JMP-FanJin}, who obtained global bounded classical solutions for system (\ref{1.0.9}) with non-degenerate diffusion, and global weak solutions for system (\ref{1.0.9}) with degenerate diffusion. Moreover, the obtained solutions converge to $(\frac{1}{|\Omega|} \int_{\Omega} u_0,0)$ as $t \rightarrow \infty$.

When $n \geqslant 2$, $D(u) \geqslant c_D(1+u)^{m-1}$ and $S(u) \leqslant c_S(1+u)^{q-1}$, where $m \in \mathbb R$ and $q \in \mathbb R$, Wang et al. \cite{2016-ZAMP-WangMuHu} demonstrated that if $q<m+\frac{n+2}{2 n}$, the classical solution of system (\ref{1.0.9}) with homogeneous Neumann boundary conditions is globally bounded; if $q<\frac{m}{2}+\frac{n+2}{2 n}$, the solutions of system (\ref{1.0.9}) exist globally. There are also some results about the global existence of solutions for system (\ref{1.0.9}) with logistic source under homogeneous Neumann boundary conditions  \cite{2013-EJDE-WangKhanKhan, 2015-BVP-Li,2017-DCDS-LankeitWang,2018-AA-WangMuHuZheng}.

For the radially symmetric no-flux/Dirichlet problem of system (\ref{1.0.9}) with $D(u)=S(u,v)=1$, Lankeit and Winkler \cite{2022-N-LankeitWinkler} found that there are globally bounded classical solutions when $n=2$. For $n\in \left\{3,4,5\right\}$, they further constructed global weak solutions. Yang and Ahn \cite{2024-NARWA-YangAhn} considered the parabolic-elliptic system (\ref{1.0.2}) with $D(u)=1$ and $S(u,v)=S(v)$, where $S(v)$ may allow singularities at $v= 0$, and established the global existence and boundedness of radial large data solutions when $n \geqslant 2$. 

Recently, the following repulsion-consumption system
\begin{align}\label{1.0.15}
  \left\{
  \begin{array}{ll}
     u_t=\nabla \cdot(D(u) \nabla u)+\nabla \cdot\left(uS(u,v) \nabla v\right),  \\
     0=\Delta v-u v,  
  \end{array}
  \right.
\end{align}
along with the no-flux/Dirichlet boundary conditions, has been studied by some authors, where $\Omega$ is a ball in $\mathbb R^{n}$ and $D(u)$ extends the prototypical choice in $D(u)=(1+u)^{-\alpha}$. For system (\ref{1.0.15}) with $S(u,v)=\frac{1}{v}$ and $n\geqslant 2$, Wang and Winkler \cite{2023-PRSESA-WangWinkler} found that, for initial data in a significantly large set of radial functions on $\Omega$, the corresponding problem possesses a finite-time blow-up solution when $\alpha>0$. For $n=2$ and $S(u,v)=1$, the explosion critical parameter for system (\ref{1.0.15}) was discovered by Ahn and Winkler \cite{2023-CVPDE-AhnWinkler}. They proved that when $\alpha>0$, for each initial data $u_0$, it is possible to identify $M_{\star}(u_0)>0$ such that there exists a classical solution blowing up in finite time whenever the boundary signal level $M>M_{\star}(u_0)$; conversely, when $\alpha \leqslant 0$, for each initial data and $M>0$, there exists a global bounded classical solution. For $n=2$ and $S$ being a rotation matrix with some $\theta \in (0,2\pi]$, Dong et al. \cite{2024-AML-DongZhangZhang} proved that the corresponding system (\ref{1.0.15}) with $\alpha>0$ admits a finite-time blow-up solution.

When the chemotactic sensitivity is linear, the results from \cite{ 2023-CVPDE-AhnWinkler, 2023-PRSESA-WangWinkler} indicate that an inhibitory effect of the diffusion is necessary in system (\ref{1.0.15}) for the occurrence of blow-up phenomenon. Inspired by this, we consider system (\ref{1.0.0}) with linear diffusion to determine an explosion critical parameter related to the nonlinear chemotactic sensitivity function.

\textbf{Main results}.
Suppose that the chemotactic sensitivity $S$ and initial data $u_0$ respectively satisfy 
\begin{align}\label{1.0.3}
  S\in C^2([0, \infty)) \text { is such that } S(\xi)\geqslant0 \text { for } \xi \geqslant 0 \text {, }
\end{align}
and 
\begin{align}\label{1.0.2}
  u_0 \in W^{1, \infty}(\Omega) \text { is radially symmetric and nonnegative with } u_0 \not \equiv 0.
\end{align}
For the parabolic-parabolic problem, we also assume that, 
\begin{align}\label{v 0}
 \tau=1 \ and \ v_0 \in W^{1, \infty}(\Omega) \text { is positive in } \bar{\Omega} \text { and radially symmetric with } v_0=M \text { on } \partial \Omega .
\end{align}

We give two preliminary propositions addressing the local existence, uniqueness and extensibility of the classical solution to (\ref{1.0.0}) with $\tau=0$ or $\tau=1$, which can be proved through a direct adaptation of the existence theory from \cite{2019-JDE-TaoWinkler}.
\begin{prop}\label{0.0}
Let $n \geqslant 1$, $R>0$ and $\Omega=B_R(0) \subset \mathbb{R}^n$. Assume that $S(\xi)$ and $u_0$ satisfy $(\ref{1.0.3})$ and $(\ref{1.0.2})$ respectively. Then there exist $T_{\max} \in(0,+\infty]$ and a unique radially symmetric pair $(u, v)$ which solves $(\ref{1.0.0})$ with $\tau=0$ in the classical sense in $\Omega \times\left(0, T_{\max }\right)$, satisfying
\begin{align*}
 \left\{
  \begin{array}{ll}
    u \in \bigcup_{q>n} C^0\left(\left[0, T_{\max }\right) ; W^{1, q}(\Omega)\right) \cap C^{2,1}\left(\bar{\Omega} \times\left(0, T_{\max }\right)\right)   \quad and \\
    v \in C^{2,0}\left(\bar{\Omega} \times\left(0, T_{\max }\right)\right).
 \end{array}
  \right.
\end{align*}
In addition, $u>0$, $v>0$ in $\bar{\Omega} \times\left(0, T_{\max }\right)$ and 
\begin{align*}
  \text { if } \, T_{\max }<\infty, \ \text{then}
  \ \limsup _{t \rightarrow T_{\max }}\|u(\cdot, t)\|_{L^{\infty}\left(\Omega\right)}=\infty.
\end{align*}
Moreover we have
\begin{align}\label{mass-0}
  \int_{\Omega} u(x, t) \mathrm{d} x=\int_{\Omega} u_0(x) \mathrm{d} x , 
  \quad  t \in\left(0, T_{\max }\right) 
\end{align}
and
\begin{align}\label{1.0.7}
  v(x, t) \leqslant M,  
  \quad  (x,t) \in \Omega \times \left(0, T_{\max }\right).
\end{align}
\end{prop}
\begin{prop}\label{0.0-1}
Let $n \geqslant 1$, $R>0$ and $\Omega=B_R(0) \subset \mathbb{R}^n$. Assume that $S(\xi)$ satisfies $(\ref{1.0.3})$. Suppose that $u_0$ and $v_0$ satisfy $(\ref{1.0.2})$ and $(\ref{v 0})$ respectively. Then there exist $T_{\max} \in(0,+\infty]$ and a unique radially symmetric pair $(u, v)$ which solves $(\ref{1.0.0})$ with $\tau=1$ in the classical sense in $\Omega \times\left(0, T_{\max }\right)$, satisfying
\begin{align*}
   \left\{\begin{array}{l}
u \in C^0(\bar{\Omega} \times[0, T_{\max})) \cap C^{2,1}(\bar{\Omega} \times(0, T_{\max})) \quad \text { and } \\
v \in \bigcap_{q>n} C^0\left([0, T_{\max}) ; W^{1, q}(\Omega)\right) \cap C^{2,1}(\bar{\Omega} \times(0, T_{\max})).
\end{array}\right.
\end{align*}
In addition, $u>0$, $v>0$ in $\bar{\Omega} \times\left(0, T_{\max }\right)$ and 
\begin{align*}
  \text { if } \, T_{\max }<\infty, \ \text{then}
  \ \limsup _{t \rightarrow T_{\max }}\|u(\cdot, t)\|_{L^{\infty}\left(\Omega\right)}=\infty.
\end{align*}
Moreover we have
\begin{align}\label{mass-1}
  \int_{\Omega} u(x, t) \mathrm{d} x=\int_{\Omega} u_0(x) \mathrm{d} x , 
  \quad  t \in\left(0, T_{\max }\right) 
\end{align}
and
\begin{align}\label{vmax}
  \|v(\cdot, t)\|_{L^{\infty}(\Omega)} \leqslant \|v_0\|_{L^{\infty}(\Omega)},
   \quad t \in (0,T_{\max}).
\end{align}
\end{prop}
Now we state our main results. 
\begin{thm}\label{0.3}
Let $\Omega=B_R(0) \subset \mathbb{R}^n$ $(n \geqslant 2)$ with $R>0$ and $\tau=\{0,1\}$. Assume that $S(\xi)$ satisfies $(\ref{1.0.3})$ and
\begin{align}\label{1.0.5-1}
  0 \leqslant S(\xi) \leqslant  K(1+\xi)^{\beta}, 
  \quad \xi \geqslant 0
\end{align}
with some $\beta \in (0, \frac{n+2}{2n})$ and $K>0$. Suppose $(\ref{1.0.2})$ and $(\ref{v 0})$ are valid. Then the corresponding classical solution $(u,v)$ of $(\ref{1.0.0})$ is globally bounded, i.e.,
\begin{align*}
  \|u(\cdot, t)\|_{L^{\infty}\left(\Omega \right)} \leqslant C, 
  \quad  t>0
\end{align*}
with some $C>0$ depending on $K$, $R$ and $\beta$.
\end{thm}
\begin{thm}\label{0.1}
Let $\Omega=B_{R}(0) \subset \mathbb R^{2}$ with $R>0$ and $\tau=0$. Assume that $S(\xi)$ satisfies $(\ref{1.0.3})$ and
\begin{align}\label{1.0.4}
  S(\xi)\geqslant k \xi^{\beta}, 
  \quad \xi \geqslant 0
\end{align}
with some $ \beta>1$ and $k>0$. Given any $u_0$ satisfying $(\ref{1.0.2})$, there exists a constant $M^{\star}=M^{\star}\left(u_0\right)>0$ such that if $M\geqslant M^{\star}$, the corresponding classical solution $(u, v)$ of $(\ref{1.0.0})$ blows up in finite time, i.e., $T_{\max}<\infty$, and 
\begin{align*}
  \limsup _{t \nearrow T_{\max }}\|u(\cdot, t)\|_{L^{\infty}(\Omega)} = \infty .
\end{align*}
\end{thm}
\begin{rem}
Due to the limitation of the method, the finite-time blow-up in the parameter region $\beta>\frac{n+2}{2n}$ for $n \geqslant 3$ is left open.
\end{rem}
\textbf{The challenge of detecting blow-up.}
Applying the method of \cite{2023-CVPDE-AhnWinkler, 2023-PRSESA-WangWinkler}, we transform the system (\ref{1.0.0}) into a single parabolic equation of the mass distribution function. We introduce a moment-type function $\phi$ (defined by $(\ref{2.3.0.1})$) and our aim is to establish a superlinear differential inequality for $\phi$ to detect finite-time blow-up. Due to the nonlinear chemotactic sensitivity, dealing with $w w_s^\beta$ is the main issue in this problem. Since the signal is degraded rather than produced, we cannot easily prove $w_{ss} \leqslant 0$ as done in \cite[Lemma 2.2]{2018-N-Winkler} to deal with this issue. To overcome this difficulty, we define an auxiliary function $\psi(t)$ in (\ref{2.3.0.2}), which is different from that in \cite{2023-CVPDE-AhnWinkler, 2023-PRSESA-WangWinkler} in establishing the differential inequality, and build up a relationship between $w$ and $\psi$ in Lemma~\ref{lem-2.3.1}.

The rest of the paper is organized as follows. The purpose of Section~\ref{section 3} is to obtain the global bounded classical solutions for system (\ref{1.0.0}) with $\tau=\{0,1\}$ and $0<\beta < \frac{n+2}{2n}$ under radial assumptions. 
In Section~\ref{section 1}, if $\Omega=B_R(0) \subset \mathbb R^2$ and $\tau=0$, we demonstrate the finite-time blow-up phenomenon under the conditions that $\beta>1$ and $M$ sufficiently large, which implies that $\beta=1$ is optimal when $\tau=0$ and $n=2$.
\section{Boundedness when $\tau=\{0,1\}$ and $0<\beta <\frac{n+2}{2n}$ }\label{section 3}
In this section, we shall prove Theorem~\ref{0.3}. We first recall a Gronwall's lemma.
\begin{lem}\label{lem-3.1.0.2}
Let $t_0 \in \mathbb{R}, T \in\left(t_0, \infty\right]$, $h>0$ and $b>0$. Suppose that the nonnegative function $g \in$ $ L_{\mathrm{loc}}^1(\mathbb{R})$ satisfies
\begin{align*}
\frac{1}{h} \int_t^{t+h} g(s) \mathrm{d} s \leqslant b, \quad t \in\left(t_0, T\right) \text {. }
\end{align*}
Then for any $a>0$ we have
\begin{align*}
\int_{t_0}^t e^{-a(t-s)} g(s) \mathrm{d} s \leqslant \frac{b h}{1-e^{-a h}}, \quad t \in\left[t_0, T\right) \text {. }
\end{align*}
Consequently, if $y \in C\left(\left[t_0, T\right)\right) \cap C^1\left(\left(t_0, T\right)\right)$ satisfies
\begin{align*}
y^{\prime}(t)+a y(t) \leqslant g(t), \quad t \in\left(t_0, T\right),
\end{align*}
then
\begin{align*}
y(t) \leqslant y\left(t_0\right)+\frac{b h}{1-e^{-a h}}, \quad t \in\left[t_0, T\right).
\end{align*}
\end{lem}
\begin{proof}
Its detailed proof can be found in \cite[Lemma 3.4]{2019-JFA-Winkler}.
\end{proof}
The following two lemmas give some basic observations on $v$. 
\begin{lem}\label{lem-3.1.0}
Let  $\Omega=B_R(0) \subset \mathbb{R}^n$ $(n \geqslant 2)$ and $(u,v)$ be the classical solution of problem $(\ref{1.0.0})$ with $\tau=0$. Suppose that $S$ and $u_0$ satisfy $(\ref{1.0.3})$ and $(\ref{1.0.2})$ respectively, then there exist constants $C_{1}>0$ and $C_2=C_2(R)>0$ such that 
\begin{align}\label{nabla v}
  \|\nabla v(\cdot, t)\|_{L^{2}(\Omega)} \leqslant C_1,
  \quad t \in (0,T_{max})
\end{align}
and
\begin{align}\label{vr-0}
\left|v_r(r, t)\right| \leqslant C_{2},
 \quad (r,t) \in \Big(\frac{4R}{5},R\Big) \times \left(0, T_{\max }\right).
\end{align}
\end{lem}
\begin{proof}
Multiplying the second equation in (\ref{1.0.0}) by $(M-v)$ and integrating by parts, using (\ref{mass-0}) and (\ref{1.0.7}), we see that 
\begin{align*}
  \int_{\Omega} |\nabla v|^2 \mathrm{d}x
  \leqslant \int_{\Omega} uv(M-v) \mathrm{d}x
 \leqslant M^2 \int_{\Omega} u \mathrm{d}x
 =M^2 \int_{\Omega} u_0 \mathrm{d}x,
 \quad t \in (0,T_{max}),
\end{align*}
which implies (\ref{nabla v}). It follows from the second equation in (\ref{1.0.0}) that
\begin{align*}
v_r
=&r^{1-n} \int_0^r \rho^{n-1} uv \mathrm{d} \rho
\leqslant r^{1-n}M \int_0^R \rho^{n-1} u \mathrm{d} \rho \notag \\
=&r^{1-n}M \frac{1}{\omega_n}\int_{\Omega} u_0 \mathrm{d} x
\leqslant \Big(\frac{4R}{5}\Big)^{1-n} M \frac{1}{\omega_n}\int_{\Omega} u_0 \mathrm{d} x,
 \quad (r,t) \in \Big(\frac{4R}{5},R\Big) \times \left(0, T_{\max }\right),
\end{align*}
which implies (\ref{vr-0}).
\end{proof}
\begin{lem}\label{heat semigroup}
Let  $\Omega=B_R(0) \subset \mathbb{R}^n$ $(n \geqslant 2)$ and $(u,v)$ be the classical solution of  problem $(\ref{1.0.0})$ with $\tau=1$, then, for each $\sigma \in\left[1, \frac{n}{n-1}\right)$, $\alpha \in (1, +\infty)$ and $r_0 \in (0,R)$, one can find constants $C_3=C_3(\sigma)>0$ and $C_4=C_4(\alpha,{r_0})>0$ such that
\begin{align}\label{nablavs}
  \|\nabla v(\cdot, t)\|_{L^{ \sigma}(\Omega)} \leqslant C_3,
  \quad t \in (0,T_{\max})
\end{align}
and 
\begin{align}\label{vrq}
\left\| v_r\right\|_{L^\alpha\left(({r_0}, R)\right)}  \leqslant C_4,
 \quad t \in (0,T_{\max}).
\end{align}
\end{lem}
\begin{proof}
For fixed $1 \leqslant {p_0} \leqslant p_1 \leqslant \infty$, according to well-known smoothing estimates for the Dirichlet heat semigroup $\left(e^{t \Delta}\right)_{t \geqslant 0}$ on $\Omega$ (\cite{1970-TMMO-EidelmanIvasishen}, \cite[Section 48.2]{2019-QuittnerSouplet}), there exist positive constants $\lambda$, $C_{5}$ and $C_{6}$ such that 
\begin{align}\label{Delta-u_0}
\left\|\nabla e^{t \Delta} \varphi\right\|_{L^{p_1}(\Omega)} 
\leqslant C_5\|\varphi\|_{W^{1, \infty}(\Omega)} ,
\quad  \varphi \in W_0^{1, \infty}(\Omega) 
\end{align}
and
\begin{align}\label{Delta-u}
\left\|\nabla e^{t \Delta} \varphi\right\|_{L^{p_1}(\Omega)} 
\leqslant C_6 \cdot\left(1+t^{-\frac{1}{2}-\frac{n}{2}(\frac{1}{{p_0}}-\frac{1}{p_1})}\right) e^{-\lambda t}\|\varphi\|_{L^{p_0}(\Omega)}, 
\quad  \varphi \in C_0(\bar{\Omega}).
\end{align}
For fixed $\sigma \in [1,\frac{n}{n-1})$, thanks to (\ref{mass-1}), (\ref{Delta-u_0}) and (\ref{Delta-u}), we apply a variation-of-constants representation to $v$ to obtain
\begin{align*}
\left\|\nabla v(\cdot, t)\right\|_{L^\sigma(\Omega)} 
 =&\|\nabla\left(v(\cdot, t)-M\right)\|_{L^\sigma(\Omega)} \notag \\
 =&\left\|\nabla e^{t \Delta}\left(v_0-M\right)-\int_0^t \nabla e^{(t-s) \Delta}\big(u(\cdot, s) v(\cdot, s)\big) \mathrm{d} s\right\|_{L^{\sigma}(\Omega)} \notag \\
 \leqslant& C_5 \left\|v_0-M\right\|_{W^{1, \infty}(\Omega)} \notag \\
            &+ C_6 \|v_0\|_{L^{\infty}(\Omega)} \|u_0\|_{L^{1}(\Omega)} \int_0^{\infty}\left(1+t^{-\frac{1}{2}-\frac{n}{2}+\frac{n}{2\sigma}}\right) e^{-\lambda t} \mathrm{d} t,
\quad t \in (0,T_{\max}),
\end{align*}
which implies (\ref{nablavs}). 

Choosing $\chi(r) \in C^{\infty}([0, R])$ such that $0 \leqslant \chi \leqslant 1$, $\chi \equiv 0$ in $\left[0, \frac{r_0}{2}\right]$ and $\chi \equiv 1$ in $[r_0, R]$, and using the second equation in (\ref{1.0.0}), we have
\begin{align}\label{chiv-M}
\big(\chi(r) \left(v-M\right)\big)_t=\big(\chi(r)\left(v-M\right)\big)_{r r}+b(r, t),
 \quad (r,t) \in(0, R) \times \left(0, T_{\max}\right),
\end{align}
where
\begin{align}\label{b}
b(r, t) = & \left(\frac{n-1}{r} \chi(r)-2 \chi_r(r)\right) v_{r}(r, t)-\chi_{r r}(r) \left(v(r, t)-M\right) \notag \\
& -\chi(r)  u(r, t) v(r, t),
 \quad (r,t) \in(0, R) \times \left(0, T_{\max }\right) .
\end{align}
Due to Proposition~\ref{0.0-1} and (\ref{nablavs}), we can deduce that $b(r,t)\in L^{\infty}\left((0,T_{\max});L^1(\frac{r_0}{2},R)\right)$. According to the known regularity of the Dirichlet heat semigroup $\left(e^{t \Delta}\right)_{t \geqslant 0}$ in one dimension, using (\ref{chiv-M}) and $\chi \cdot (v-M)=0$ on $\left\{\frac{r_0}{2}, R\right\} \times\left(0, T_{\max}\right)$, we can find $C_7 = C_7 > 0$ such that whenever $\alpha \in (1, \infty)$,
\begin{align}\label{vr semilinear estimate}
 \left\| v_r\right\|_{L^\alpha\left(\left(r_0, R\right)\right)} 
 \leqslant & \left\|\partial_r\big(\chi \cdot\left(v(\cdot, t)-M\right)\big)\right\|_{L^{\alpha}\left(\left(\frac{r_0}{2}, R\right)\right)} \notag \\ 
 =& \left\|\partial_r e^{t \Delta}\big(\chi \cdot\left(v_0-M\right)\big)+\int_0^t \partial_r e^{(t-s) \Delta} b(\cdot, s) \mathrm{d} s\right\|_{L^{\alpha}\left(\left(\frac{r_0}{2}, R\right)\right)} \notag \\
 \leqslant & C_5\left\|\chi \cdot\left(v_0-M\right)\right\|_{W^{1, \infty}\left(\left(\frac{r_0}{2}, R\right)\right)} \notag \\
            &+ C_6 \int_0^t\left(1+(t-s)^{-1+\frac{1}{2\alpha}}\right) e^{-\lambda(t-s)}\left\|b(\cdot, s)\right\|_{L^1\left(\left(\frac{r_0}{2}, R\right)\right)} \mathrm{d} s \notag \\
 \leqslant & C_7,
  \quad t \in (0,T_{\max}).
\end{align}
We complete our proof.
\end{proof}
The following lemma provides an estimate for an integral of $u(R,t)$ with respect to $t$.
\begin{lem}\label{uR-time-all}
Let $\Omega=B_R(0) \subset \mathbb{R}^n$ $(n \geqslant 2)$ and $(u,v)$ be the classical solution of problem $(\ref{1.0.0})$ with $\tau=\{0,1\}$. Assume that $(\ref{1.0.2})$ and $(\ref{v 0})$ are valid. Suppose that $S(\xi)$ satisfies $(\ref{1.0.3})$ and $(\ref{1.0.5-1})$ with $0<\beta<\frac{n+2}{2n}$ and $K>0$, then there exists a constant $C_8=C_8(\beta,R,K)$ such that 
\begin{align}\label{uR}
\int_t^{t+h} u(R, s) \mathrm{d} s \leqslant C_8,
\quad t \in (0,T_{\max}-h),
\end{align}
where $h=\min \left\{1, \frac{1}{2} T_{\max}\right\}$.
\end{lem}
\begin{proof}
We fix $\zeta \in C^{\infty}(\bar{\Omega})$ such that $0 \leqslant \zeta \leqslant 1$, $\zeta \equiv 0$ in $\bar{B}_{\frac{R}{4}}(0)$ and $\zeta \equiv 1$ in $\bar{\Omega} \backslash \bar{B}_{\frac{R}{2}(0)}$. Multiplying the first equation in (\ref{1.0.0}) by $\zeta^2 (1+u)^{a-1}$ $(0<a<1)$ and integrating by parts, we find that
\begin{align}\label{zeta2up}
  &\frac{1}{{a}} \frac{\mathrm{d}}{\mathrm{d} t} \int_{\Omega} \zeta^2 (1+u)^{a} \mathrm{d}x \notag \\
   =& \int_{\Omega} \zeta^2 (1+u)^{{a}-1} (\Delta u + \nabla \cdot(S(u)\nabla v) \mathrm{d}x \notag \\
 =& (1-{a}) \int_{\Omega} \zeta^2 (1+u)^{{a}-2}\left|\nabla u\right|^2 \mathrm{d}x
    +(1-{a}) \int_{\Omega} \zeta^2 (1+u)^{{a}-2}S(u) \nabla u \cdot \nabla v \mathrm{d}x \notag \\
& -2 \int_{\Omega} \zeta (1+u)^{{a}-1} \nabla u \cdot \nabla \zeta \mathrm{d}x
    -2 \int_{\Omega} \zeta (1+u)^{a-1} S(u) \nabla v \cdot \nabla \zeta \mathrm{d}x,
    \quad t \in (0,T_{\max}).
\end{align}
For the three terms on the right side of (\ref{zeta2up}), we apply Young's inequality and (\ref{1.0.5-1}) to deduce that 
\begin{align}\label{nablau nablav}
& \left| (1-{a}) \int_{\Omega} \zeta^2 (1+u)^{{a}-2}S(u) \nabla u \cdot \nabla v \mathrm{d}x\right| \notag \\
 \leqslant &  \frac{(1-{a})}{4} \int_{\Omega} \zeta^2 (1+u)^{{a}-2}\left|\nabla u\right|^2 \mathrm{d}x 
          + (1-a)K^2 \int_{\Omega} {\zeta}^2 (1+u)^{a+2\beta-2} |\nabla v|^2 \mathrm{d}x
\end{align}
and
\begin{align}\label{nablau nablazeta}
&\left| -2 \int_{\Omega} \zeta (1+u)^{{a}-1} \nabla u \cdot \nabla \zeta \mathrm{d}x \right| \notag \\
\leqslant & \frac{1-{a}}{4}  \int_{\Omega} \zeta^2 (1+u)^{{a}-2}\left|\nabla u\right|^2 \mathrm{d}x 
           + \frac{4}{1-{a}} \int_{\Omega} (1+u)^{a} |\nabla \zeta|^2  \mathrm{d}x ,
\end{align}
as well as
\begin{align}\label{nablav nablazeta}
&\left| -2 \int_{\Omega} \zeta (1+u)^{a-1} S(u) \nabla v \cdot \nabla \zeta \mathrm{d}x \right| \notag \\
 \leqslant & K^2\int_{\Omega} {\zeta}^2 (1+u)^{a+2\beta-2} |\nabla v|^2 \mathrm{d}x
            + \int_{\Omega} (1+u)^{a} |\nabla \zeta|^2  \mathrm{d}x .
\end{align}
For one of the same terms on the right side of (\ref{nablau nablav}), (\ref{nablau nablazeta}) and (\ref{nablav nablazeta}), we use Hölder's inequality to obtain
\begin{align}\label{up nabla zeta2}
\int_{\Omega} (1+u)^{a} |\nabla \zeta|^2  \mathrm{d}x
  \leqslant  \Big(\int_{\Omega} 1+u \mathrm{d}x\Big)^{a} 
             \cdot \Big(\int_{\Omega}|\nabla \zeta|^{\frac{2}{1-{a}}} \mathrm{d}x \Big)^{1-{a}}.
\end{align}
In the following, we divide the regions of the parameters $\beta$ and $a$ into three cases to derive the lemma.

\textit{Case 1}: $0<\beta\leqslant\frac{1}{2}$ and $0<a<1$. Thus, we have $a+2\beta-2<0$. Due to (\ref{nabla v}) and (\ref{vrq}) with $r_0=\frac{R}{4}$ and $\alpha=2>1$, along with $u>0$, we infer that
\begin{align}\label{zeta2 up+2beta-2 nablav2-1}
& \int_{\Omega} {\zeta}^2 (1+u)^{a+2\beta-2} |\nabla v|^{2} \mathrm{d}x 
\leqslant  \int_{\Omega} {\zeta}^2 |\nabla v|^{2} \mathrm{d}x
\leqslant \max\{C_1^2, \omega_n R^{n-1}C_3^2 \}.
\end{align}
Summing up (\ref{zeta2up})-(\ref{zeta2 up+2beta-2 nablav2-1}), there exists a constant $C_9=C_9(a,\beta,R,K)>0$ such that
\begin{align*}
 \frac{1-{a}}{2}\int_{\Omega} \zeta^2 (1+u)^{{a}-2}\left|\nabla u\right|^2 \mathrm{d}x
    \leqslant &  \frac{1}{{a}} \frac{\mathrm{d}}{\mathrm{d} t} \int_{\Omega} \zeta^2 (1+u)^{a} \mathrm{d}x
               +\Big(\frac{4}{1-{a}}+1\Big)  \int_{\Omega} (1+u)^{a} |\nabla \zeta|^2  \mathrm{d}x \notag \\
              &+ (2-a)K^2 \int_{\Omega} {\zeta}^2 (1+u)^{a+2\beta-2} |\nabla v|^{2} \mathrm{d}x \notag \\
    \leqslant &  \frac{1}{{a}} \frac{\mathrm{d}}{\mathrm{d}t} \int_{\Omega} \zeta^2 (1+u)^{a} \mathrm{d}x +C_9,
      \quad t \in (0,T_{\max}).
\end{align*} 
Integrating this differential equation on $(t,t+h)$ for all $t \in (0,T_{\max}-h)$ with $h=\min \left\{1, \frac{1}{2} T_{\max}\right\}$, and using (\ref{mass-0}) and (\ref{mass-1}), we derive that there exist positive constants $C_{10}=C_{10}(R)$ and $C_{11}=C_{11}(a,\beta,R,K)$ such that
\begin{align}\label{uR-time-1}
& \int_t^{t+h} \int_{\frac{R}{2}}^R\left|\left((1+u)^{\frac{{a}}{2}}\right)_r \right|^2 \mathrm{d}r \mathrm{d}s \notag \\
  \leqslant & C_{10} \int_{t}^{t+h}\int_{\Omega} \zeta^2 (1+u)^{{a}-2}\left|\nabla u\right|^2 \mathrm{d}x \mathrm{d}s \notag \\
  \leqslant &  \frac{2C_{10}}{{a}(1-{a})} \int_{\Omega} \zeta^2 \big(1+u(\cdot, t+h)\big)^{a}  \mathrm{d}x
             -\frac{2C_{10}}{{a}(1-{a})} \int_{\Omega} \zeta^2 \big(1+u(\cdot, t)\big)^{a}  \mathrm{d}x  +\frac{2C_9}{1-a}h \notag \\
  \leqslant &   \frac{2C_{10}|\Omega|^{1-{a}}}{{a}(1-{a})} \Big(\int_{\Omega} 1+u\left(\cdot, t+h\right) \mathrm{d}x\Big)^{a} +\frac{2C_9}{1-a} \notag \\
  \leqslant & C_{11}.
\end{align}
It follows from (\ref{uR-time-1}) with $a=\frac{2}{3}$ and Gagliardo-Nirenberg inequality to find a positive constant $C_{12}=C_{12}(R)$ such that  
\begin{align}\label{uRcase1}
\int_t^{t+h} u(R, s) \mathrm{d} s 
 \leqslant & \int_t^{t+h} \big(1+u(R, s)\big)^{\frac{5}{3}} \mathrm{d} s  \notag \\
 \leqslant & \int_t^{t+h} \|(1+u)^{\frac{1}{3}}(\cdot,s)\|_{L^{\infty}\left((\frac{R}{2},R)\right)}^{5} \mathrm{d} s  \notag \\
 \leqslant & C_{12}\int_t^{t+h} \|\big((1+u)^{\frac{1}{3}}\big)_r(\cdot,s)\|_{L^2\left((\frac{R}{2},R)\right)}^2  \|(1+u)^{\frac{1}{3}}(\cdot,s)\|_{L^3\left((\frac{R}{2},R)\right)}^3 \mathrm{d} s \notag \\ 
             & + C_{12}\int_t^{t+h} \|(1+u)^{\frac{1}{3}}(\cdot,s)\|_{L^3\left((\frac{R}{2},R)\right)}^5 \mathrm{d} s \notag \\
  =& C_{12}(|\Omega|+\|u_0\|_{L^1(\Omega)})C_{11}
     + C_{12}(|\Omega|+\|u_0\|_{L^1(\Omega)})^{\frac{5}{3}}
\end{align}
for all $t\in [0,T_{\max}-h)$ with $h=\min \left\{1, \frac{1}{2} T_{\max}\right\}$.

\textit{Case 2}: $\frac{1}{2}<\beta<\frac{n+2}{2n}$ and $0<a \leqslant 2-2\beta$. Thus, we have $a+2\beta-2 \leqslant 0$. Similarly, we also obtain
\begin{align}\label{uR-time-2}
\int_t^{t+h} \int_{\frac{R}{2}}^R\left|\left((1+u)^{\frac{{a}}{2}}\right)_r \right|^2 \mathrm{d}r \mathrm{d}s
\leqslant C_{11}.
\end{align}
Taking $0<a=1-\beta \leqslant 2-2\beta$ in (\ref{uR-time-2}), and upon application of Gagliardo-Nirenberg inequality, there exists constant $C_{13}=C_{13}(R)>0$ such that
\begin{align}\label{uRcase2}
\int_t^{t+h} u(R, s) \mathrm{d} s
  \leqslant & \int_t^{t+h} \big(1+u(R, s)\big)^{2-\beta} \mathrm{d} s \notag \\
  \leqslant & \int_t^{t+h} \|(1+u)^{\frac{1-\beta}{2}}(\cdot,s)\|_{L^{\infty}\left((\frac{R}{2},R)\right)}^{\frac{4-2\beta}{1-\beta}} \mathrm{d} s \notag \\
  \leqslant & C_{13}\int_t^{t+h} \|\big((1+u)^{\frac{1-\beta}{2}}\big)_r(\cdot,s)\|_{L^2\left((\frac{R}{2},R)\right)}^2
               \|(1+u)^{\frac{1-\beta}{2}}(\cdot,s)\|_{L^{\frac{2}{1-\beta}}\left((\frac{R}{2},R)\right)}^{\frac{2}{1-\beta}} \mathrm{d} s \notag \\
             & +C_{13} \int_t^{t+h} \|(1+u)^{\frac{1-\beta}{2}}(\cdot,s)\|_{L^{\frac{2}{1-\beta}}\left((\frac{R}{2},R)\right)}^{\frac{4-2\beta}{1-\beta}} \mathrm{d} s \notag \\
  =&C_{13}(|\Omega|+\|u_0\|_{L^1(\Omega)})C_{11}
    + C_{13}(|\Omega|+\|u_0\|_{L^1(\Omega)})^{2-\beta}
\end{align}
for all $t\in [0,T_{\max}-h)$ with $h=\min \left\{1, \frac{1}{2} T_{\max}\right\}$.

\textit{Case 3}: $\frac{1}{2}<\beta<\frac{n+2}{2n}$ and $2-2\beta<a<1$. Thus, we have $0<a+2\beta-2<1$. Using Hölder's inequality,  (\ref{vr-0}) and (\ref{vrq}) with $\frac{2}{3-2\beta-{a}}>0$ and $r_0=\frac{R}{4}$, for $\tau=\{0,1\}$, we can find a positive constant $C_{14}=C_{14}(a,\beta,R)$ such that
\begin{align*}
& \int_{\Omega} {\zeta}^2 (1+u)^{a+2\beta-2} |\nabla v|^{2} \mathrm{d}x \notag \\
  \leqslant & \left(\int_{\Omega} 1+u \mathrm{d}x \right)^{a+2\beta-2} 
             \cdot \left(\int_{\Omega}{\zeta}^{\frac{2}{3-2\beta-{a}}} |\nabla v|^{\frac{2}{3-2\beta-{a}}} \mathrm{d}x \right)^{3-2\beta-{a}} \notag \\
  \leqslant & \left(\int_{\Omega} 1+u \mathrm{d}x \right)^{a+2\beta-2} 
             \cdot (\omega_n R^{n-1})^{3-2\beta-{a}} \left(\int_{\frac{R}{4}}^R | v_r|^{\frac{2}{3-2\beta-{a}}} \mathrm{d}x \right)^{3-2\beta-{a}} \notag \\
  \leqslant & C_{14}.
\end{align*}
By means of this, and similar to the proof of (\ref{uR-time-1}), we also have
\begin{align}\label{uR-time-3}
\int_t^{t+h} \int_{\frac{R}{2}}^R\left|\left((1+u)^{\frac{{a}}{2}}\right)_r \right|^2 \mathrm{d}r \mathrm{d}s
\leqslant C_{11}.
\end{align}
We utilize (\ref{uR-time-3}) with $2-2\beta<a=\frac{3-2\beta}{2}<1$, along with Gagliardo-Nirenberg inequality, to find a constant $C_{15}=C_{15}(R)>0$ such that 
\begin{align}\label{uRcase3}
\int_t^{t+h} u(R, s) \mathrm{d} s
  \leqslant & \int_t^{t+h} \big(1+u(R, s)\big)^{\frac{5}{2}-\beta} \mathrm{d} s \notag \\
  \leqslant & \int_t^{t+h} \|(1+u)^{\frac{3-2\beta}{4}}(\cdot,s)\|_{L^{\infty}\left((\frac{R}{2},R)\right)}^{\frac{10-4\beta}{3-2\beta}} \mathrm{d} s \notag \\
  \leqslant & C_{15}\int_t^{t+h} \|\big((1+u)^{\frac{3-2\beta}{4}}\big)_r(\cdot,s)\|_{L^2\left((\frac{R}{2},R)\right)}^2
               \|(1+u)^{\frac{3-2\beta}{4}}(\cdot,s)\|_{L^{\frac{4}{3-2\beta}}\left((\frac{R}{2},R)\right)}^{\frac{4}{3-2\beta}} \mathrm{d} s \notag \\
             & + C_{15}\int_t^{t+h} \|(1+u)^{\frac{3-2\beta}{4}}(\cdot,s)\|_{L^{\frac{4}{3-2\beta}}\left((\frac{R}{2},R)\right)}^{\frac{10-4\beta}{3-2\beta}} \mathrm{d} s \notag \\
  =&C_{15}(|\Omega|+\|u_0\|_{L^1(\Omega)})C_{11}
    + C_{15}(|\Omega|+\|u_0\|_{L^1(\Omega)})^{\frac{5}{2}-\beta}
\end{align}
for all $t \in [0,T_{\max}-h)$ with $h=\min \left\{1, \frac{1}{2} T_{\max}\right\}$. 

Therefore, (\ref{uRcase1}) and (\ref{uRcase2}), together with (\ref{uRcase3}), imply (\ref{uR}). Moreover, for any $\beta \in (0,\frac{n+2}{2n})$ and $a\in (0,1)$, 
\begin{align}\label{uR-time-all}
\int_t^{t+h} \int_{\frac{R}{2}}^R\left|\left((1+u)^{\frac{{a}}{2}}\right)_r \right|^2 \mathrm{d}r \mathrm{d}s
\leqslant C_{11},
\quad t \in [0,T_{\max}-h),
\end{align}
as a consequence of (\ref{uR-time-1}), (\ref{uR-time-2}) and (\ref{uR-time-3}).
\end{proof}
The following lemma provides a spatio-temporal uniform bound for $v_r$ near the boundary.
\begin{lem}\label{vr-all}
Let $\Omega=B_R(0) \subset \mathbb{R}^n$ $(n \geqslant 2)$ and $(u,v)$ be the classical solution of problem $(\ref{1.0.0})$ with $\tau=1$. Suppose that $S(\xi)$ satisfies $(\ref{1.0.3})$ and $(\ref{1.0.5-1})$ with $\beta \in (0,\frac{n+2}{2n})$ and $K>0$, then for any radially symmetric $u_0$ and $v_0$ satisfy $(\ref{1.0.2})$ and $(\ref{v 0})$ respectively, there exists a constant $C_{16}=C_{16}(\beta,R,K)$ such that 
\begin{align}\label{vr}
\left|v_r(r, t)\right| \leqslant C_{16},
 \quad (r,t) \in \Big(\frac{4R}{5},R\Big) \times \left(0, T_{\max }\right).
\end{align}
\end{lem}
\begin{proof}
We choose $\zeta \in C^{\infty}([0, R])$ such that $0 \leqslant \zeta \leqslant 1$, $\zeta \equiv 0$ in $\bar{B}_{\frac{R}{2}}(0)$ and $\zeta \equiv 1$ in $\bar{\Omega} \backslash B_\frac{3R}{4}(0)$. Multiplying the first equation in (\ref{1.0.0}) by $\zeta^2 (1+u)$, integrating by parts, and then applying Young's inequality, we find that
\begin{align}\label{vr zeta2 up}
  &\frac{1}{{2}} \frac{\mathrm{d}}{\mathrm{d} t} \int_{\Omega} \zeta^2 (1+u)^{2} \mathrm{d}x 
  + \int_{\Omega} \zeta^2 (1+u)^{2} \mathrm{d}x \notag \\
  = & - \int_{\Omega} \zeta^2 \left|\nabla u\right|^2 \mathrm{d}x
      - \int_{\Omega} \zeta^2 S(u) \nabla u \cdot \nabla v \mathrm{d}x \notag \\
  & -2 \int_{\Omega} \zeta (1+u) \nabla u \cdot \nabla \zeta \mathrm{d}x
    -2 \int_{\Omega} \zeta (1+u)S(u)\nabla v \cdot \nabla \zeta \mathrm{d}x
    + \int_{\Omega} \zeta^2 (1+u)^{2} \mathrm{d}x \notag \\
  \leqslant & \frac{3K^2}{2} \int_{\Omega} \zeta^2 (1+u)^{2\beta} |\nabla v|^{2} \mathrm{d}x 
             +3 \int_{\Omega} (1+u)^{2}  |\nabla \zeta|^2 \mathrm{d}x \notag \\
             & + \int_{\Omega} \zeta^2 (1+u)^{2} \mathrm{d}x,
  \quad t \in (0,T_{\max}).
\end{align}

For $a\in (0,1)$, due to $0<\beta<\frac{n+2}{2n}$, we have $\frac{2+a}{2\beta}>1$. This enable us to apply Young's inequality to obtain
\begin{align}\label{zeta2 up nablav2-4alpha}
\int_{\Omega} \zeta^2 (1+u)^{2\beta} |\nabla v|^{2} \mathrm{d}x 
\leqslant & \int_{\Omega \backslash B_{\frac{R}{2}}(0)} (1+u)^{{2+{a}}} \mathrm{d}x
          +\int_{\Omega \backslash B_{\frac{R}{2}}(0)}\left|\nabla v\right|^{\frac{2(2+a)}{2+a-2\beta}} \mathrm{d}x \notag \\
\leqslant & \int_{\Omega \backslash B_{\frac{R}{2}}(0)} (1+u)^{{2+{a}}} \mathrm{d}x +C_{17},
\end{align}
where (\ref{vrq}) with $r_0=\frac{R}{2}$ and $\alpha=\frac{2(2+a)}{2+a-2\beta}$ is used. Again using Young's inequality, we have
\begin{align}\label{0}
\int_{\Omega} (1+u)^2  |\nabla \zeta|^2 \mathrm{d}x
  \leqslant \int_{\Omega \backslash B_{\frac{R}{2}}(0)} (1+u)^{2+{a}} \mathrm{d}x
            + \int_{\Omega \backslash B_{\frac{R}{2}}(0)} \left|\nabla \zeta \right|^{\frac{2(2+{a})}{{{a}}}} \mathrm{d}x
\end{align}
and
\begin{align}\label{zeta2 up}
\int_{\Omega} \zeta^2 (1+u)^2 \mathrm{d}x
  \leqslant \int_{\Omega \backslash B_{\frac{R}{2}}(0)} (1+u)^{{2+{a}}} \mathrm{d}x
            +  \int_{\Omega \backslash B_{\frac{R}{2}}(0)} \zeta^{\frac{2(2+{a})}{{{a}}}} \mathrm{d}x.
\end{align}
Substituting (\ref{zeta2 up nablav2-4alpha})-(\ref{zeta2 up}) into (\ref{vr zeta2 up}) implies the existence of a positive constant $C_{18}=C_{18}(a,\beta,R)$ such that
\begin{align}\label{zeta2 up-ineq}
&\frac{1}{2} \frac{\mathrm{d}}{\mathrm{d} t} \int_{\Omega} \zeta^2 (1+u)^2 \mathrm{d}x 
+ \int_{\Omega} \zeta^2 (1+u)^2 \mathrm{d}x \notag \\
 \leqslant & \left(\frac{3K^2}{2}+4\right)\int_{\Omega \backslash B_{\frac{R}{2}}(0)} (1+u)^{{2+{a}}}(x,t) \mathrm{d}x +C_{18} \notag \\ \overset{\Delta}{=}& g(t) ,
 \quad t \in (0,T_{\max}).
\end{align}
We estimate
\begin{align}\label{u}
\|(1+u)^{\frac{a}{2}}\|_{L^{\frac{2}{a}}\left((\frac{R}{2},R)\right)}^{\frac{2}{a}} 
  = & \int_{\frac{R}{2}}^{R} \big(1+u(r,t)\big) \mathrm{d}r \notag \\ 
  \leqslant & \Big(\frac{2}{R}\Big)^{n-1} \int_{\frac{R}{2}}^{R} r^{n-1}\big(1+u(r,t)\big) \mathrm{d}r \notag \\ 
  \leqslant & \Big(\frac{2}{R}\Big)^{n-1}\frac{1}{\omega_n} \int_{\Omega} (1+u_0) \mathrm{d}x.
\end{align}
By Gagliardo-Nirenberg inequality, there exists a positive constant $C_{19}=C_{19}(a,R)$ such that
\begin{align*}
 \int_{\Omega \backslash B_{\frac{R}{2}}(0)} (1+u)^{{2+{a}}} \mathrm{d}x
  \leqslant & \omega_n R^{n-1} \int_{\frac{R}{2}}^{R} \big(1+u(r,t)\big)^{{2+{a}}} \mathrm{d}r \notag \\
  =&\omega_n R^{n-1} \|(1+u)^{\frac{a}{2}}\|_{L^{\frac{2(2+a)}{a}}\left((\frac{R}{2},R)\right)}^{\frac{2(2+a)}{a}} \notag \\
  \leqslant & C_{19}\|\big((1+u)^{\frac{a}{2}}\big)_r\|_{L^2\left((\frac{R}{2},R)\right)}^2  \|(1+u)^{\frac{a}{2}}\|_{L^{\frac{2}{a}}\left((\frac{R}{2},R)\right)}^{\frac{4}{a}} \notag \\
             &+ C_{19}\|(1+u)^{\frac{a}{2}}\|_{L^{\frac{2}{a}}\left((\frac{R}{2},R)\right)}^{\frac{2(2+a)}{a}}.
\end{align*}
 Combining this with (\ref{uR-time-all}) and (\ref{u}) guarantees the existence of a positive constant $C_{20}=C_{20}(a,\beta,R,K)>0$ such that
\begin{align}\label{u2+p1}
\int_t^{t+h}\int_{\Omega \backslash B_{\frac{R}{2}}(0)} (1+u)^{{2+{a}}} \mathrm{d}x \mathrm{d}s \leqslant C_{20}
\end{align}
for all $t \in [0,T_{\max}-h)$ with $h=\min \left\{1, \frac{1}{2} T_{\max}\right\}$. Combining (\ref{u2+p1}) with (\ref{zeta2 up-ineq}), and applying Lemma~\ref{lem-3.1.0.2} to (\ref{zeta2 up-ineq}), we have
\begin{align}\label{u2}
\int_{\frac{3R}{4}}^R u^2 \mathrm{d}r 
  \leqslant \Big(\frac{4}{3R}\Big)^{n-1}\frac{1}{\omega_n}\int_{\Omega}(1+u_0)^2\mathrm{d}x + \Big(\frac{4}{3R}\Big)^{n-1}\frac{1}{\omega_n}\frac{2C_{21}}{(1-e^{-2})},
\quad t \in (0,T_{\max}),
\end{align}
where 
\begin{align*}
C_{21}=C_{21}(a,\beta,R,K):= \sup _{t \in\left(0, T_{\max}-h\right)} \int_t^{t+h} g(s) \mathrm{d} s
\end{align*}
due to (\ref{u2+p1}) and $g(t)$ is defined in (\ref{zeta2 up-ineq}). 

We fix $\chi \in C^{\infty}([0, R])$ such that $0 \leqslant \chi \leqslant 1$, $\chi \equiv 0$ in $\left[0, \frac{3R}{4}\right]$ and $\chi \equiv 1$ in $[\frac{4R}{5}, R]$. For $b(x,t)$ defined in (\ref{b}), along with (\ref{u2}) and (\ref{vr semilinear estimate}), we have the estimate 
\begin{align}\label{b2}
\left\|b(\cdot, t)\right\|_{L^2 \left((\frac{3R}{4}, R)\right)} \leqslant C_{22},
\quad t \in (0,T_{\max})
\end{align}
for some constant $C_{22}=C_{22}(\beta,R,K)>0$. Thus, by (\ref{Delta-u_0}) and (\ref{Delta-u}), we have
\begin{align*}
  \left\|v_r(\cdot, t)\right\|_{L^{\infty}\left((\frac{4R}{5}, R)\right)} 
  \leqslant &\left\|\partial_r\big(\chi \cdot\left(v(\cdot, t)-M\right)\big)\right\|_{L^{\infty}\left((\frac{3R}{4}, R)\right)} \notag \\
  \leqslant & C_5\left\|\chi \cdot\left(v_0-M\right)\right\|_{W^{1, \infty}\left((\frac{3R}{4}, R)\right)}+C_6C_{22} \int_0^{\infty}\Big(1+t^{-\frac{3}{4}}\Big) e^{-\lambda t} \mathrm{d} t
\end{align*}
for all $t \in (0,T_{\max})$, which implies (\ref{vr}).
\end{proof}
\begin{lem}\label{p-q}
Let $\Omega=B_R(0) \subset \mathbb{R}^n$ $(n \geqslant 3)$ and $(u,v)$ be the classical solution of problem $(\ref{1.0.0})$ with $\tau=\{0,1\}$. Suppose that $S(\xi)$ satisfies $(\ref{1.0.3})$ and $(\ref{1.0.5-1})$ with $\beta \in (0,\frac{n+2}{2n})$ and $K>0$. Assume that $(\ref{1.0.2})$ and $(\ref{v 0})$ are valid. For $q > \max \left\{1, \frac{n-2}{2}\right\}$ and $p \in \left[2 q+2, \frac{(2 n-2) q}{n-2}+1\right]$, there exists a constant $C_{23}>0$ such that 
\begin{align}\label{p-n}
\|\nabla v\|_{L^p(\Omega)} 
  \leqslant C_{23}\left\||\nabla v|^{q-1}  |D^2v| \right\|_{L^2(\Omega)}^{\frac{p-2}{p q}}
            +C_{23}.
\end{align}
If $n=2$, for $q \geqslant 1$ and $p \geqslant 2 q+2$, we also have
\begin{align}\label{p-2}
\|\nabla v\|_{L^p(\Omega)} 
  \leqslant C_{23}\left\||\nabla v|^{q-1}  |D^2v| \right\|_{L^2(\Omega)}^{\frac{p-2}{p q}}
            +C_{23}.
\end{align}
\end{lem}
\begin{proof}
Due to (\ref{vr-0}), (\ref{vr}), $p>1$ and the condition $v = M$ on $\partial \Omega$, for all $0<\beta<\frac{n+2}{2n}$ and $t \in (0,T_{\max})$, there exists a constant $C_{24}>0$ such that
\begin{align}\label{partialv-1}
\int_{\partial \Omega} |\nabla v|^{p-2}v\frac{\partial v}{\partial n} \mathrm{d}x
  \leqslant |\partial \Omega| M \left\|v_ r\right\|_{L^{\infty}\left((\frac{4R}{5}, R) \times\left(0, T_{\max }\right)\right)}^{p-1}
  \leqslant C_{24}.
\end{align}
Utilizing (\ref{partialv-1}), (\ref{1.0.7}) and (\ref{vmax}), and following a proof similar to the one used in Lemma 3.3 of \cite{2015-ZAMP-WangXiang}, we obtain (\ref{p-n}). Combining (\ref{partialv-1}) with (\ref{nablavs}), (\ref{nabla v}), (\ref{1.0.7}) and (\ref{vmax}), and similar to \cite[Lemma 2.1]{2015-MMMAS-LiSuenWinklerXue}, we deduce (\ref{p-2}).
\end{proof}
To derive the boundedness of $\int_{\Omega} u^p \mathrm{d}x$, we first establish following differential inequality.
\begin{lem}\label{lem-3.1.2}
Let $\Omega=B_R(0) \subset \mathbb{R}^n$ $(n \geqslant 2)$ and $(u,v)$ be the classical solution of problem $(\ref{1.0.0})$ with $\tau=\{0,1\}$. Suppose that $S(\xi)$ satisfies $(\ref{1.0.3})$ and $(\ref{1.0.5-1})$ with $\beta \in (0,\frac{n+2}{2n})$ and $K>0$. Assume that $(\ref{1.0.2})$ and $(\ref{v 0})$ are valid. Then for any 
\begin{align}\label{m}
m>\max\{1,\frac{n-2}{2},\frac{2n\beta-n}{2n-2n\beta+2}\}
\end{align}
and 
\begin{align}\label{p}
\max \{1,\frac{m}{m+1}+2-2\beta,m+1-\frac{2}{n},\frac{2(n-2)(m+1)(\beta-1)}{2m+2-n}\}<p<(m+1)(2-2\beta)+\frac{2m}{n},
\end{align}
 there exist constants $C_{25}=C_{25}(\beta,R,m,p,K)>0$ and $C_{26}=C_{26}(\beta,R,m,p,K)>0$ such that 
\begin{align}\label{all-2}
 & \frac{\mathrm{d}}{\mathrm{d} t} \left(\frac{\tau}{2m}\int_{\Omega} |\nabla v|^{2m} \mathrm{d}x 
   + \frac{1}{p} \int_{\Omega}(1+u)^p \mathrm{d}x \right)  
   + C_{25} \int_{\Omega} |\nabla v|^{2(m-1)} |D^2 v|^2 \mathrm{d}x\notag \\
  &  +  C_{25}\int_{\Omega} (1+u)^{p-2}|\nabla u|^2 \mathrm{d}x   \notag \\
\leqslant & \frac{1}{2} \int_{\partial \Omega} |\nabla v|^{2(m-1)} \frac{\partial |\nabla v|^2}{\partial n} \mathrm{d}x 
             + \int_{\partial \Omega} |\nabla v|^{2(m-1)} uv \left|\frac{\partial v}{\partial n}\right| \mathrm{d}x  +C_{26}
\end{align}
for all $t \in (0,T_{\max})$.
\end{lem}
\begin{proof}
We first verify (\ref{p}) using (\ref{m}), which justifies our choice of $p$. Due to $m>\frac{2n\beta-n}{2n-2n\beta+2}>\frac{2n\beta-2-n}{2n-2n\beta+2}$ and $\beta<\frac{n+2}{2n}<\frac{n+1}{n}$, we have
\begin{align}\label{fixp-1}
&\frac{m}{m+1}+2-2\beta-\Big((m+1)(2-2\beta)+\frac{2m}{n}\Big) \notag \\
=&\frac{m}{m+1}-m\Big(2-2\beta+\frac{2}{n}\Big) \notag \\
=&\frac{m}{m+1}\Big(1-\big(2-2\beta+\frac{2}{n}\big)(m+1)\Big)  \notag \\
<&\frac{m}{m+1}\Big(1-\big(2-2\beta+\frac{2}{n}\big)\big(\frac{2n\beta-2-n}{2n-2n\beta+2}+1\big)\Big) \notag \\
=&0
\end{align}
and 
\begin{align}\label{fixp-3}
&\frac{2(n-2)(m+1)(\beta-1)}{2m+2-n}-\Big((m+1)(2-2\beta)+\frac{2m}{n}\Big)  \notag \\
= &(m+1)\frac{4m(\beta-1)}{2m+2-n}-\frac{2m}{n} \notag \\
=&2m\frac{\big(m(2n\beta-2n-2)-(2+n-2n\beta)\big)}{n(2m+2-n)} \notag \\
<&2m\frac{\Big(\frac{2n\beta-2-n}{2n-2n\beta+2}(2n\beta-2n-2)-(2+n-2n\beta)\Big)}{n(2m+2-n)} \notag \\
=&0.
\end{align}
Since $\beta<\frac{n+2}{2n}$, we deduce that
\begin{align}\label{fixp-2}
(m+1-\frac{2}{n})-\Big((m+1)(2-2\beta)+\frac{2m}{n}\Big)
=&(m+1)\Big(2\beta-1-\frac{2}{n}\Big) \notag \\
<&(m+1)\Big(\frac{n+2}{n}-1-\frac{2}{n}\Big)\notag \\
=&0.
\end{align}
According to $\beta<\frac{n+2}{2n}$ and $m>\frac{2n\beta-n}{2n-2n\beta+2}$, we have
\begin{align}\label{fixp-4}
(m+1)(2-2\beta)+\frac{2m}{n}-1=m\Big(2-2\beta+\frac{2}{n}\Big)+1-2\beta>0.
\end{align}
Combining (\ref{fixp-1})-(\ref{fixp-4}), (\ref{p}) is justified.

Our choice $p>m+1-\frac{2}{n}>(1-\frac{2}{n})m+1-\frac{2}{n}$ entails that
\begin{align}\label{theta1<1}
\Big(\frac{p}{2(m+1)}-\frac{p}{2}\Big)-\Big(\frac{1}{2}-\frac{1}{n}-\frac{p}{2}\Big)
>\frac{(1-\frac{2}{n})(m+1)}{2(m+1)}-\Big(\frac{1}{2}-\frac{1}{n}\Big)
=0.
\end{align}
Due to $p>1$, we deduce that
\begin{align}\label{<0}
\frac{1}{2}-\frac{1}{n}-\frac{p}{2}<-\frac{1}{n}<0.
\end{align}
According to $p>\frac{m}{m+1}+2-2\beta$, we infer that
\begin{align*}
p-2+2\beta>\frac{m}{m+1}>0
\end{align*}
and
\begin{align}\label{theta2>0}
\frac{pm}{2(m+1)(p-2+2\beta)}-\frac{p}{2}
=&\frac{p}{2}\left(\frac{m-(m+1)(2\beta-2)-p(m+1)}{(m+1)(p-2+2\beta)}\right) \notag \\
<&\frac{p}{2}\left(\frac{m-(m+1)(2\beta-2)-(\frac{m}{m+1}+2-2\beta)(m+1)}{(m+1)(p-2+2\beta)}\right)\notag \\
=&0.
\end{align}
Using $p>\frac{2(n-2)(m+1)(\beta-1)}{2m+2-n}$, we deduce that
\begin{align}\label{theta2<1}
&\Big(\frac{pm}{2(m+1)(p-2+2\beta)}-\frac{p}{2}\Big)-\Big(\frac{1}{2}-\frac{1}{n}-\frac{p}{2}\Big) \notag \\
=&\frac{p(\frac{2}{n}(m+1)-1)-(m+1)(2\beta-2)(1-\frac{2}{n})}{2(m+1)(p-2+2\beta)} \notag \\
>&0.
\end{align}

Multiplying the first equation in (\ref{1.0.0}) by $(1+u)^{p-1}$($p>1$) and integrating by parts, using Young's inequality, we obtain
\begin{align}\label{(1+u)p}
\frac{1}{p}\frac{\mathrm{d}}{\mathrm{d}t} \int_{\Omega}(1+u)^p \mathrm{d}x 
 = & \int_{\Omega} (1+u)^{p-1} (\Delta u+\nabla \cdot (S(u) \nabla v)) \mathrm{d}x \notag \\
 = & -(p-1)\int_{\Omega} (1+u)^{p-2} |\nabla u|^2 \mathrm{d}x
     -(p-1)\int_{\Omega} (1+u)^{p-2} S(u) \nabla u \cdot \nabla v \mathrm{d}x \notag \\
 \leqslant & -\frac{p-1}{2} \int_{\Omega} (1+u)^{p-2} |\nabla u|^2 \mathrm{d}x
             + \frac{p-1}{2} \int_{\Omega} (1+u)^{p-2} S^2(u) |\nabla v|^2 \mathrm{d}x \notag \\
 \leqslant & -\frac{p-1}{2} \int_{\Omega} (1+u)^{p-2} |\nabla u|^2 \mathrm{d}x
             + \frac{(p-1)K^2}{2} \int_{\Omega} (1+u)^{p-2+2\beta} |\nabla v|^2 \mathrm{d}x 
\end{align}
for all $ t \in (0,T_{\max})$. 

For any $m>1$, it follows from the second equation in (\ref{1.0.0}) that
\begin{align}\label{v2k}
\frac{\tau}{2m} \frac{\mathrm{d}}{\mathrm{d}t} \int_{\Omega} |\nabla v|^{2m} \mathrm{d}x 
   = &\int_{\Omega}  |\nabla v|^{2(m-1)} \nabla v \cdot \nabla (\Delta v-uv) \mathrm{d}x  \notag \\
   = &\int_{\Omega}  |\nabla v|^{2(m-1)} \nabla v \cdot \nabla \Delta v \mathrm{d}x 
      - \int_{\Omega}  |\nabla v|^{2(m-1)} \nabla v \cdot \nabla (uv) \mathrm{d}x
\end{align}
for all $ t \in (0,T_{\max})$. For the first term in the right side of (\ref{v2k}), we using the fact that $\nabla v \cdot \nabla \Delta v= \frac{1}{2} \Delta |\nabla v|^2-|D^2v|^2$ to obtain
\begin{align}\label{v2-4alpha-1}
  \int_{\Omega}  |\nabla v|^{2(m-1)} \nabla v \cdot \nabla \Delta v \mathrm{d}x 
     =& \frac{1}{2} \int_{\Omega}  |\nabla v|^{2(m-1)} \Delta |\nabla v|^2 \mathrm{d}x 
       - \int_{\Omega}  |\nabla v|^{2(m-1)} |D^2v|^2 \mathrm{d}x  \notag \\
     =&-\frac{m-1}{2} \int_{\Omega} |\nabla v|^{2(m-2)} |\nabla |\nabla v|^2|^2 \mathrm{d}x \notag \\
      &+ \frac{1}{2} \int_{\partial \Omega} |\nabla v|^{2(m-1)} \frac{\partial |\nabla v|^2}{\partial n} \mathrm{d}x 
      -\int_{\Omega} |\nabla v|^{2(m-1)} |D^2 v|^2 \mathrm{d}x .
\end{align}
For the second term in the right side of (\ref{v2k}), integrating by parts, we get
\begin{align}\label{v2-4alpha-2}
&\int_{\Omega} |\nabla v|^{2(m-1)} \nabla v \cdot \nabla (uv) \mathrm{d}x \notag \\
  =& -\int_{\Omega} |\nabla v|^{2(m-1)} uv \Delta v \mathrm{d}x  
    -(m-1) \int_{\Omega} |\nabla v|^{2(m-2)} uv \nabla |\nabla v|^2 \cdot \nabla v \mathrm{d}x  \notag \\
    &+ \int_{\partial \Omega} |\nabla v|^{2(m-1)} uv \frac{\partial v}{\partial n} \mathrm{d}x.
\end{align}
 By means of Young's inequality, we find 
\begin{align}\label{v2-4alpha-2-1}
\int_{\Omega} |\nabla v|^{2(m-1)} uv \Delta v \mathrm{d}x
  \leqslant & \frac{1}{2n} \int_{\Omega} |\nabla v|^{2(m-1)} |\Delta v|^2 \mathrm{d}x 
            + \frac{n}{2}\int_{\Omega} |\nabla v|^{2(m-1)} u^2 v^2 \mathrm{d}x \notag \\
  \leqslant & \frac{1}{2} \int_{\Omega} |\nabla v|^{2(m-1)} |D^2 v|^2 \mathrm{d}x 
            + \frac{n}{2}\int_{\Omega} |\nabla v|^{2(m-1)} u^2 v^2 \mathrm{d}x 
\end{align}
and
\begin{align}\label{v2-4alpha-2-2}
& (m-1)\int_{\Omega} |\nabla v|^{2(m-2)} uv \nabla |\nabla v|^2 \cdot \nabla v \mathrm{d}x \notag \\
\leqslant & \frac{m-1}{4}\int_{\Omega} |\nabla v|^{2(m-2)} |\nabla |\nabla v|^2|^2 \mathrm{d}x
            +(m-1)  \int_{\Omega} |\nabla v|^{2(m-1)} u^2 v^2 \mathrm{d}x .
\end{align}
Substituting (\ref{v2-4alpha-1})-(\ref{v2-4alpha-2-2}) into (\ref{v2k}) yields 
\begin{align*}
&\frac{\tau}{2m} \frac{\mathrm{d}}{\mathrm{d}t}\int_{\Omega} |\nabla v|^{2m} \mathrm{d}x 
+\frac{m-1}{4} \int_{\Omega} |\nabla v|^{2(m-2)} |\nabla |\nabla v|^2|^2 \mathrm{d}x
+\frac{1}{2}\int_{\Omega} |\nabla v|^{2(m-1)} |D^2 v|^2 \mathrm{d}x  \notag \\
  \leqslant &  \frac{1}{2} \int_{\partial \Omega} |\nabla v|^{2(m-1)} \frac{\partial |\nabla v|^2}{\partial n} \mathrm{d}x 
             - \int_{\partial \Omega} |\nabla v|^{2(m-1)} uv \frac{\partial v}{\partial n} \mathrm{d}x
             +(m-1+\frac{n}{2}) \int_{\Omega} |\nabla v|^{2(m-1)} u^2 v^2 \mathrm{d}x
\end{align*}
for all $t \in (0,T_{\max})$. Combining this with (\ref{(1+u)p}), applying the fact $|\nabla|\nabla v|^2|^2=4|\nabla v|^2|D^2 v|^2$, and using (\ref{1.0.7}) and (\ref{vmax}), we deduce that
\begin{align}\label{all-1}
&\frac{\mathrm{d}}{\mathrm{d}t}\left(\frac{\tau}{2m}\int_{\Omega} |\nabla v|^{2m} \mathrm{d}x 
   + \frac{1}{p} \int_{\Omega}(1+u)^p \mathrm{d}x \right) \notag \\
  & +\frac{p-1}{2}\int_{\Omega} (1+u)^{p-2}|\nabla u|^2 \mathrm{d}x 
   +(m-\frac{1}{2})\int_{\Omega} |\nabla v|^{2(m-1)} |D^2 v|^2 \mathrm{d}x  \notag \\
\leqslant &  \frac{1}{2} \int_{\partial \Omega} |\nabla v|^{2(m-1)} \frac{\partial |\nabla v|^2}{\partial n} \mathrm{d}x 
             + \int_{\partial \Omega} |\nabla v|^{2(m-1)} uv \left|\frac{\partial v}{\partial n}\right| \mathrm{d}x  \notag \\
            & +\frac{(p-1)K^2}{2} \int_{\Omega} (1+u)^{p-2+2\beta} |\nabla v|^{2}  \mathrm{d}x
              +\Big(m-1+\frac{n}{2}\Big)\|v_0\|_{L^{\infty}(\Omega)}^2 \int_{\Omega} u^2 |\nabla v|^{2(m-1)} \mathrm{d}x
\end{align}
for all $t \in (0,T_{\max})$. 

By means of Hölder's inequality, we have
\begin{align}\label{u2 nabla v2-4alpha}
  \int_{\Omega} u^2 |\nabla v|^{2(m-1)} \mathrm{d}x
    \leqslant \left(\int_{\Omega}|\nabla v|^{2(m+1)}\mathrm{d}x\right)^{\frac{m-1}{m+1}}
              \left(\int_{\Omega}(1+u)^{m+1}\mathrm{d}x\right)^{\frac{2}{m+1}}.
\end{align}
Applying Lemma~\ref{p-q} with $q=m$ and $p=2m+2$, we have
\begin{align}\label{v-lemma p-q}
 \left(\int_{\Omega}|\nabla v|^{2m+2}\mathrm{d}x\right)^{\frac{m-1}{m+1}} 
  = \big\|\nabla v\big\|_{L^{2m+2}(\Omega)}^{2(m-1)}
  \leqslant C_{23} \left\||\nabla v|^{m-1} | D^2v| \right\|_{L^2(\Omega)}^{\frac{2(m-1)}{m+1}}+C_{23}.
\end{align}
By Gagliardo-Nirenberg inequality, there exists a constant $C_{27}=C_{27}(p,m)>0$ such that
\begin{align}\label{u3-2alpha-G-N}
  \left(\int_{\Omega}(1+u)^{m+1}\mathrm{d}x\right)^{\frac{2}{m+1}}
  =&\|(1+u)^{\frac{p}{2}}\|_{L^{\frac{2(m+1)}{p}}(\Omega)}^{\frac{4}{p}}   \notag \\
  \leqslant & C_{27}\|\nabla (1+u)^{\frac{p}{2}}\|_{L^2(\Omega)}^{\frac{4\theta_1}{p}} 
              \|(1+u)^{\frac{p}{2}}\|_{L^{\frac{2}{p}}(\Omega)}^{\frac{4(1-\theta_1)}{p}} \notag \\
             & +C_{27}\|(1+u)^{\frac{p}{2}}\|_{L^{\frac{2}{p}}(\Omega)}^{\frac{4}{p}},
\end{align}
where $\theta_1= \frac{\frac{p}{2(m+1)}-\frac{p}{2}}{\frac{1}{2}-\frac{1}{n}-\frac{p}{2}} \in (0,1)$ due to (\ref{theta1<1}) and (\ref{<0}). Thanks to $p>m+1-\frac{2}{n}$ and (\ref{<0}), we have
\begin{align}\label{<1-1}
\frac{2\theta_1}{p}+\frac{m-1}{m+1} -1
=&\frac{2}{p} \cdot \frac{\frac{p}{2(m+1)}-\frac{p}{2}}{\frac{1}{2}-\frac{1}{n}-\frac{p}{2}}+\frac{m-1}{m+1}-1 \notag \\
=&\frac{-\frac{m}{m+1}+\frac{m-1}{m+1}(\frac{1}{2}-\frac{1}{n}-\frac{p}{2})}{\frac{1}{2}-\frac{1}{n}-\frac{p}{2}}-1 \notag \\
<&\frac{-\frac{m}{m+1}+\frac{m-1}{m+1}(\frac{1}{2}-\frac{1}{n}-(\frac{m}{2}+\frac{1}{2}-\frac{1}{n}))}{\frac{1}{2}-\frac{1}{n}-\frac{p}{2}}-1 \notag \\
=&\frac{-\frac{m}{2}}{\frac{1}{2}-\frac{1}{n}-\frac{p}{2}}-1 \notag \\
=&\frac{\frac{p}{2}-\frac{1}{2}+\frac{1}{n}-\frac{m}{2}}{\frac{1}{2}-\frac{1}{n}-\frac{p}{2}} \notag \\
<&0.
\end{align}
Inserting (\ref{v-lemma p-q}) and (\ref{u3-2alpha-G-N}) into (\ref{u2 nabla v2-4alpha}), due to (\ref{<1-1}), we can apply Young's inequality to deduce that, for any $\varepsilon_1>0$, there exists a positive constant $C_{28}=C_{28}(\beta,p,R,m,\varepsilon_1)$ such that 
\begin{align}\label{nablav2k-2 u2}
\int_{\Omega} u^2 |\nabla v|^{2(m-1)}  \mathrm{d}x 
\leqslant \varepsilon_1  \left\||\nabla v|^{m-1} | D^2v| \right\|_{L^2(\Omega)}^2
          + \varepsilon_1 \|\nabla (1+u)^{\frac{p}{2}}\|_{L^2(\Omega)}^2 +C_{28}.
\end{align}

For the last but one term in (\ref{all-1}), using Hölder's inequality, we obtain
\begin{align}\label{complex}
\int_{\Omega} (1+u)^{p-2+2\beta} |\nabla v|^{2}  \mathrm{d}x
\leqslant \left(\int_{\Omega} |\nabla v|^{2m+2} \mathrm{d}x \right)^{\frac{1}{m+1}} 
           \left(\int_{\Omega} (1+u)^{\frac{(m+1)(p-2+2\beta)}{m}} \mathrm{d}x \right)^{\frac{m}{m+1}}.
\end{align}
By Gagliardo-Nirenberg inequality, we infer the existence of a constant $C_{29}=C_{29}(p,\beta,m)>0$ such that
\begin{align}\label{u-complex}
\left(\int_{\Omega} (1+u)^{\frac{(m+1)(p-2+2\beta)}{m}} \mathrm{d}x \right)^{\frac{m}{m+1}}
=&\|(1+u)^\frac{p}{2}\|_{L^{\frac{2(m+1)(p+2\beta-2)}{pm}}(\Omega)}^{\frac{2(p+2\beta-2)}{p}} \notag \\
\leqslant & C_{29}\| \nabla (1+u)^{\frac{p}{2}}\|_{L^2(\Omega)}^{\frac{2\theta_2(p+2\beta-2)}{p}}
           \|(1+u)^{\frac{p}{2}}\|_{L^{\frac{2}{p}}(\Omega)}^{\frac{2(1-\theta_2)(p+2\beta-2)}{p}} \notag \\
           &+ C_{29} \|(1+u)^{\frac{p}{2}}\|_{L^{\frac{2}{p}}(\Omega)}^{\frac{2(p+2\beta-2)}{p}},
\end{align}
where $\theta_2=\frac{\frac{pm}{2(m+1)(p-2+2\beta)}-\frac{p}{2}}{\frac{1}{2}-\frac{1}{n}-\frac{p}{2}} \in (0,1)$ by (\ref{<0}), (\ref{theta2>0}) and (\ref{theta2<1}). Similar to (\ref{v-lemma p-q}), we find 
\begin{align}\label{nablav2k+2}
 \left(\int_{\Omega}|\nabla v|^{2m+2}\mathrm{d}x\right)^{\frac{1}{m+1}} 
 = \big\|\nabla v\big\|_{L^{2m+2}(\Omega)}^{2}
  \leqslant C_{23} \left\||\nabla v|^{m-1} | D^2v| \right\|_{L^2(\Omega)}^{\frac{2}{m+1}}+C_{23}.
\end{align}
Because of $p<(2-2\beta)(m+1)+\frac{2m}{n}$ and (\ref{<0}), we have
\begin{align}\label{<1}
&\frac{p-2+2\beta}{p}\theta_2+\frac{1}{m+1}-1 \notag \\
=&\frac{\frac{m}{2(m+1)}-\frac{p-2+2\beta}{2}}{\frac{1}{2}-\frac{1}{n}-\frac{p}{2}}+\frac{1}{m+1}-1 \notag \\
=&\frac{1-\beta+\frac{m}{n(m+1)}-\frac{p}{2(m+1)}}{\frac{1}{2}-\frac{1}{n}-\frac{p}{2}} \notag \\
<&0.
\end{align}
Inserting (\ref{u-complex}) and (\ref{nablav2k+2}) into (\ref{complex}), due to (\ref{<1}), we can apply Young's inequality to deduce that, for any $\varepsilon_2>0$, there exists a positive constant $C_{30}=C_{30}(\beta,R,m,p,\varepsilon_2)>0$ such that 
\begin{align}\label{nablav2 up-2+2beta}
\int_{\Omega} (1+u)^{p-2+2\beta} |\nabla v|^{2}  \mathrm{d}x
\leqslant \varepsilon_2  \left\||\nabla v|^{m-1} | D^2v| \right\|_{L^2(\Omega)}^2
          + \varepsilon_2 \|\nabla (1+u)^{\frac{p}{2}}\|_{L^2(\Omega)}^2 +C_{30}.
\end{align}
Substituting (\ref{nablav2k-2 u2}) and (\ref{nablav2 up-2+2beta}) into (\ref{all-1}), and choosing $\varepsilon_1$ and $\varepsilon_2$ sufficiently small, we deduce (\ref{all-2}).
\end{proof}
The following lemma gives the $L^p$ estimate for $u$. Since $v$ satisfies the Dirichlet boundary condition, inspired by \cite{2022-N-LankeitWinkler}, we use the radial symmetry assumption to handle the boundary integrals.
\begin{lem}\label{lem-3.2.3}
Let $\Omega=B_R(0) \subset \mathbb{R}^n$ $(n \geqslant 2)$ and $(u,v)$ be the classical solution of problem $(\ref{1.0.0})$ with $\tau=\{0,1\}$. Suppose that $S(\xi)$ satisfies $(\ref{1.0.3})$ and $(\ref{1.0.5-1})$ with $\beta \in (0,\frac{n+2}{2n})$ and $K>0$. Assume that $(\ref{1.0.2})$ and $(\ref{v 0})$ are valid. If $m$ and $p$ satisfy (\ref{m}) and (\ref{p}) respectively, then there exists a constant $C_{31}=C_{31}(\beta,R,m,p,K)>0$ such that 
\begin{align}\label{u2-all}
\int_{\Omega}u^p \mathrm{d}x + \tau \int_{\Omega}|\nabla v|^{2m} \mathrm{d}x \leqslant C_{31},
\quad t \in (0,T_{\max}).
\end{align}
\end{lem}
\begin{proof}
For $\tau=1$, we use (\ref{nablavs}), and apply Gagliardo-Nirenberg inequality and Young's inequality to find a constant $C_{32}=C_{32}(m)>0$ such that
\begin{align}\label{nablav4-4alpha} 
 \int_{\Omega} |\nabla v|^{2m}\mathrm{d}x 
   = & \big\| |\nabla v|^{m}\big\|_{L^2(\Omega)}^2 \notag \\
   \leqslant & C_{32} \big\|\nabla |\nabla v|^{m}\big\|_{L^2(\Omega)}^{2\theta_4}
               \big\| |\nabla v|^{m}\big\|_{L^{\frac{{\sigma}}{m}}(\Omega)}^{2(1-\theta_4)}
               + C_{32} \big\| |\nabla v|^{m}\big\|_{L^{\frac{{\sigma}}{m}}(\Omega)}^2 \notag \\
   \leqslant & C_{32}C_3^{2m(1-\theta_4)} \big\|\nabla |\nabla v|^{m}\big\|_{L^2(\Omega)}^{2\theta_4} +C_{32}C_3^{2m} \notag \\
   \leqslant & m^{2\theta_4}C_{32}C_3^{2m(1-\theta_4)} \left\||\nabla v|^{m-1} |D^2v| \right\|_{L^2(\Omega)}^{2\theta_4} +C_{32}C_3^{2m},  
\end{align}
where $\theta_4=\frac{2mn-n{\sigma}}{2{\sigma}-n{\sigma}+2mn} \in (0,1)$ by $m>1$ and ${\sigma}>1$. 
For $\tau=\{0,1\}$, using Gagliardo-Nirenberg inequality, there exists a constant $C_{33}=C_{33}(p)>0$ such that
\begin{align}\label{u2-G-N}
\int_{\Omega}(1+u)^p \mathrm{d}x 
  \leqslant & C_{33}\| \nabla (1+u)^{\frac{p}{2}}\|_{L^2(\Omega)}^{2\theta_3}   
              \|(1+u)^{\frac{p}{2}}\|_{L^{\frac{2}{p}}(\Omega)} ^{2(1-\theta_3)} 
              + C_{33}\|u(1+u)^{\frac{p}{2}}\|_{L^{\frac{2}{p}}(\Omega)}^2 ,
\end{align}
where $\theta_3=\frac{\frac{1}{2}-\frac{p}{2}}{\frac{1}{2}-\frac{1}{n}-\frac{p}{2}} \in (0,1)$ due to (\ref{<0}) and $p>1$. 

For $\tau=\{0,1\}$, using the second equation in (\ref{1.0.0}), since $v=M$ on $\partial \Omega$, we deduce that
\begin{align*}
  \frac{\partial\left|\nabla v\right|^2}{\partial \nu}
= &2 v_r v_{ r r} \notag \\
= &2 v_{ r} \cdot\left\{v_{ r r}+\frac{1}{R} v_{r}\right\}-\frac{2}{R} v_{r}^2 \notag \\
= &2 u v v_{ r}-\frac{2}{R } v_{ r}^2 \notag \\
\leqslant & 2 u v v_{ r} .
\end{align*}
Combining this with (\ref{uR}), (\ref{vr-0}), (\ref{vr}) and $m>1$, for $\tau=\{0,1\}$, we infer that
\begin{align}\label{partial nablav}
\int_t^{t+h}\int_{\partial \Omega} |\nabla v|^{2(m-1)} \frac{\partial |\nabla v|^2}{\partial n} \mathrm{d}x \mathrm{d}s
  \leqslant & 2M\int_t^{t+h} u(R,s) v_r^{2m-1}(R,s) \mathrm{d}x \mathrm{d}s \notag \\
  \leqslant & 2M\max\{C_2,C_{16}\}^{2m-1} C_8,
  \quad t \in [0,T_{\max}-h) 
\end{align}
and
\begin{align}\label{partial uv}
\int_t^{t+h}\int_{\partial \Omega} |\nabla v|^{2(m-1)} uv \left|\frac{\partial v}{\partial n}\right| \mathrm{d}x \mathrm{d}s 
\leqslant M\max\{C_2,C_{16}\}^{2m-1} C_8,
\quad t \in [0,T_{\max}-h).
\end{align}
In view of (\ref{nablav4-4alpha})-(\ref{partial uv}) and (\ref{all-2}), along with Young's inequality, for $\tau=\{0,1\}$, we deduce that
\begin{align}\label{all-3}
 & \frac{\mathrm{d}}{\mathrm{d} t} \left(\frac{\tau}{2m}\int_{\Omega} |\nabla v|^{2m} \mathrm{d}x 
   + \frac{1}{p} \int_{\Omega}(1+u)^p \mathrm{d}x \right)  
   + \frac{\tau}{2m} \int_{\Omega } |\nabla v|^{2m}\mathrm{d}x
   +\frac{1}{p} \int_{\Omega}(1+u)^p \mathrm{d}x    \notag \\
 \leqslant & \frac{1}{2} \int_{\partial \Omega} |\nabla v|^{2(m-1)} \frac{\partial |\nabla v|^2}{\partial n} \mathrm{d}x 
             + \int_{\partial \Omega} |\nabla v|^{2(m-1)} uv \left|\frac{\partial v}{\partial n}\right| \mathrm{d}x  +C_{34}
\end{align}
for all $t \in (0,T_{\max})$. We define 
\begin{align*}
y(t):=\frac{\tau}{2m}\int_{\Omega} |\nabla v|^{2m} \mathrm{d}x 
   + \frac{1}{p} \int_{\Omega}(1+u)^p \mathrm{d}x 
\end{align*}
and 
\begin{align*}
g(t):=  \frac{1}{2} \int_{\partial \Omega} |\nabla v|^{2(m-1)} \frac{\partial |\nabla v|^2}{\partial n} \mathrm{d}x 
        + \int_{\partial \Omega} |\nabla v|^{2(m-1)} uv \left|\frac{\partial v}{\partial n}\right| \mathrm{d}x  +C_{34}.
\end{align*}
Since (\ref{partial nablav}) and (\ref{partial uv}) ensures that $\int_t^{t+h}  g(s) \mathrm{d}s \leqslant C_{35}(\beta,R,K,m,p) $ for all $t \in (0,T_{\max}-h)$ with $h=\min \left\{1, \frac{1}{2} T_{\max }\right\}$. It follows from (\ref{all-3}) that
\begin{align*}
y^{\prime}(t)+ y(t) \leqslant g(t), \quad t \in (0,T_{\max}),
\end{align*}
which implies (\ref{u2-all}) by means of Lemma~\ref{lem-3.1.0.2}.
\end{proof}
Applying the standard Moser-type iterative argument, we finally obtain the $L^{\infty}$ estimate of $u$.\\
\emph{\textbf{Proof of Theorem \ref{0.3}.}} Since $(m+1)(2-2\beta)+\frac{2m}{n} \rightarrow \infty$ as $m\rightarrow \infty$, it follows from Lemma~\ref{lem-3.2.3} that, for any $p>1$ and $m>1$, we have
\begin{align}\label{up0}
\sup _{0<t<T_{\max }}\|u(\cdot, t)\|_{L^{p}(\Omega)} +\tau \|\nabla v(\cdot, t)\|_{L^{2m}(\Omega)}\leqslant C_{31}.
\end{align}
For the case of $\tau=0$, by the standard elliptic regularity theory and (\ref{up0}) with $p>n$, we have 
\begin{align}\label{nablav-0}
\left\|\nabla v(\cdot, t)\right\|_{L^{\infty}(\Omega)} \leqslant C_{36} .
\end{align}
For the case of $\tau=1$, by (\ref{up0}) with $p>n$, using the standard Dirichlet heat semigroup estimate, we find a constant $C_{36}>0$ such that
\begin{align}\label{nablav-1}
\left\|\nabla v(\cdot, t)\right\|_{L^{\infty}(\Omega)} 
= &\|\nabla\left(v(\cdot, t)-M\right)\|_{L^{\infty}(\Omega)} \notag \\
= &\left\|\nabla e^{t \Delta}\left(v_0-M\right)-\int_0^t \nabla e^{(t-s) \Delta}\big(u(\cdot, s) v(\cdot, s)\big) \mathrm{d} s\right\|_{L^{\infty}(\Omega)} \notag \\
\leqslant & C_5 \left\|v_0-M\right\|_{W^{1, \infty}(\Omega)} \notag \\
           &+C_6 \|v_0\|_{L^{\infty}(\Omega)} \|u\|_{L^{p}(\Omega)} \int_0^{\infty}\left(1+s^{-\frac{1}{2}-\frac{n}{2p}}\right) e^{-\lambda s} \mathrm{d} s \notag \\
\leqslant & C_{36}.
\quad t \in (0,T_{\max}).
\end{align} 
Thus, using (\ref{nablav-1}) or (\ref{nablav-0}), we deduce Theorem \ref{0.3} for the case of $\tau=0$ or $\tau=1$ by a Moser-type iterative argument (cf.  \cite[Lemma A.1]{2012-JDE-TaoWinkler}).
\hfill$\Box$\\
\vskip 3mm
\section{Blow-up in 2-D system when $\tau=0$ and $\beta >1$}\label{section 1}
In this section, we assume that $\Omega =B_R(0) \subset \mathbb{R}^2$ and $\tau=0$. The purpose of this section is to prove Theorem~\ref{0.1}. Now we introduce a mass distribution function, and derive some properties.
\begin{lem}\label{lem-2.1.2}
Assume that $(\ref{1.0.3})$ and $(\ref{1.0.2})$ hold. We define
\begin{align}\label{lem-2.1.2.1}
  w(s, t):=\int_0^{\sqrt{s}}\rho u \left(\rho, t \right) \mathrm{d}\rho, 
  \quad  (s,t) \in\left[0, R^2\right] \times \left[0, T_{\max }\right),
\end{align}
then $w \in C\left(\left[0, T_{\max }\right) ; C^1\left(\left[0, R^2\right]\right) \cap C^{2,1}\left(\left(0, R^2\right] \times
\left(0, T_{\max }\right)\right)\right.$ satisfies the following Dirichlet problem
\begin{align}\label{lem-2.1.2.3}
  \left\{
  \begin{array}{ll}
    w_t(s, t)=4 s w_{s s}(s, t)+ \sqrt{s} S\bigl(2w_s(s, t)\bigl) v_r\big(\sqrt{s}, t\big), 
    & s \in\left(0, R^2\right),\ t \in\left(0, T_{\max }\right), \\
    w(0, t)=0, \quad w\left(R^2, t\right)=\frac{{m_0}}{2\pi}, 
    & t \in\left[0, T_{\max }\right),\\
    w(s, 0)=w_0(s), 
    & s \in\left(0, R^2\right),
  \end{array}
  \right.
\end{align}
where ${m_0}:=\int_{\Omega} u_0 \mathrm{d} x$ and $w_0(s)=\int_0^{\sqrt{s}} \rho u_0(\rho) \mathrm{d} \rho$.

Moreover, $w$ satisfies 
\begin{align}\label{lem-2.1.2.2}
  w_s(s, t)=\frac{1}{2} u\big(\sqrt{s}, t\big), 
  \quad (s,t) \in\left[0, R^2\right] \times \left[0, T_{\max }\right).
\end{align}
\end{lem}
\begin{proof}
(\ref{lem-2.1.2.3}) can be easily verified using the radial symmetric form of (\ref{1.0.0}), we omit the details.
\end{proof}
We give a lower bound for $r v_r(r,t)$ by employing an ODE comparison argument, and its detailed proof can be found in \cite[Lemma 3.1]{2023-PRSESA-WangWinkler}. 
\begin{lem}\label{lem2.2.1}
Assume that $(\ref{1.0.3})$ and $(\ref{1.0.2})$ hold. Then 
\begin{align}\label{lem2.2.1.1}
  r v_r(r, t) \geqslant \frac{U(r, t) v(r, t)}{1+\int_0^r \frac{U \left(\rho, t\right)} {\rho} \mathrm{d} \rho}, 
  \quad (r,t) \in(0, R) \times \left(0, T_{\max }\right) \text {, }
\end{align}
where 
\begin{align*}
  U(r, t):=w \big(r^2, t\big), 
  \quad (r,t) \in[0, R] \times \left[0, T_{\max }\right). 
\end{align*}
\end{lem}
In view of Lemma~\ref{lem2.2.1}, we can estimate $v$ from below, whose proof is the same as \cite[ Lemma 2.4]{2023-CVPDE-AhnWinkler}.
\begin{lem}\label{lem-2.2.2}
Assume that $(\ref{1.0.3})$ and $(\ref{1.0.2})$ hold. Then
\begin{align}\label{lem-2.2.2.1}
  v(r, t) \geqslant M \exp \Big[-\Big(\frac{{m_0}}{2 \pi} \cdot \ln \frac{R}{r}\Big)^{\frac{1}{2}}\Big], 
  \quad (r,t) \in (0, R] \times \left(0, T_ {\max }\right).
\end{align}
\end{lem}
Applying the above lemmas, a lower bound for $w_t$ can be derived.
\begin{lem}\label{lem2.2.3}
Assume that $(\ref{1.0.3})$ and $(\ref{1.0.2})$ hold. Suppose that $S(\xi)$ satisfies $(\ref{1.0.4})$ with $\beta>1$ and $k>0$. Then, there exists a constant $C_{37}=C_{37}(\beta , R) >0$, such that
\begin{equation}\label{lem2.2.3.1}
\begin{aligned}
  w_t(s,t) \geqslant 4sw_{s s}(s,t) 
  + M kC_{37}
  \cdot \frac{s^{\frac{\beta-1}{2}} w(s,t) w_s^\beta(s,t) }{1+\frac{1}{2} \int_0^{s} \frac{ w(\rho, t)}{\rho} \mathrm{d} \rho},
  \quad (s,t) \in \left(0, R^2\right) \times \left(0, T_{\max }\right).
\end{aligned}
\end{equation}
\end{lem}
\begin{proof}
By Young's inequality, we have
\begin{align}\label{lem.2.2.3.4}
  \Big(\frac{{m_0}}{2 \pi} \cdot \ln \frac{R}{\sqrt{s}}\Big)^{\frac{1}{2}} 
  \leqslant & (\beta-1)\ln \frac{R}{\sqrt{s}}+\frac{{m_0}}{8\pi(\beta-1)} \notag\\
  = &\ln \frac{R^{\beta-1}}{s^{\frac{\beta-1}{2}}}+\frac{{m_0}}{8\pi(\beta-1)}.
\end{align}
Substituting (\ref{lem.2.2.3.4}) into (\ref{lem-2.2.2.1}) yields 
\begin{align}\label{lem-2.2.2.2}
  v(\sqrt{s}, t) \geqslant M R^{-(\beta-1)} e^{-\frac{{m_0}}{8 \pi (\beta-1)}} s^{\frac{\beta-1}{2}},  
  \quad (s,t) \in \left(0, R^2\right) \times \left(0, T_{\max }\right) \text {. }
\end{align}
Recalling the definitions of $U$ and $w$, we derive that
\begin{align}\label{lem2.2.3.3}
  \int_0^{\sqrt s} U(\rho, t) \rho^{-1} \mathrm{d} \rho
  =\frac{1}{2} \int_0^s \frac{w(\rho, t)}{\rho} \mathrm{d} \rho.
\end{align}
According to (\ref{lem2.2.1.1}), (\ref{lem-2.2.2.2}), (\ref{lem2.2.3.3}) and the positivity of $w$, we obtain
\begin{align}\label{lem2.2.3.2}
  \sqrt {s}  v_r \big(\sqrt {s} ,t\big) 
  \geqslant & M R^{-(\beta-1)} e^{-\frac{{m_0}}{8 \pi (\beta-1)}} \frac{s^{\frac{\beta-1}{2}} w(s, t) }{1+\frac{1}{2} 
    \int_0^s \frac{w(\rho, t)}{\rho} \mathrm{d} \rho},
    \quad (s,t) \in \left(0, R^2\right) \times \left(0, T_{\max }\right) \text {. }
\end{align} 
By (\ref{1.0.4}), (\ref{lem-2.1.2.3}) and (\ref{lem2.2.3.2}), there exists a positive constant $C_{37}(\beta , R)=2^\beta R^{-(\beta-1)} e^{-\frac{{m_0}}{8 \pi (\beta-1)}}$ such that (\ref{lem2.2.3.1}) holds.
\end{proof}
The approach to detect blow-up is based on a differential inequality of a moment-type functional $\phi(t)$.
For any given $\gamma=\gamma\left(\beta\right) \in(0,1)$, we define such a positive functional
\begin{align}\label{2.3.0.1}
  \phi(t):=\int_0^{R^2} s^{-\gamma} w(s, t) \mathrm{d} s, 
  \quad t \in\left[0, T_{\max }\right),
\end{align}
which is well-defined and belongs to $C\left(\left[0, T_{\max }\right)\right) \cap C^1\left(\left(0, T_{\max }\right)\right)$. In order to establish the differential inequality of $\phi(t)$, given $\gamma \in (0,1)$, we further introduce an auxiliary functional, which belongs to $ C (\left[ 0,T_{\max} \right)$), defined as follows:
\begin{align}\label{2.3.0.2}
  \psi(t):=\int_0^{R^2} s^{\frac{\beta-1}{2}-\gamma} w(s, t)w^{\beta}_s(s,t)  \mathrm{d}s, 
  \quad t \in\left[0, T_{\max }\right).
\end{align} 
\begin{lem}\label{lem-2.3.1}
Assume that $(\ref{1.0.3})$ and $(\ref{1.0.2})$ hold. Then for any $\beta \in (1,+\infty)$ and $l=l(\beta) \in (1,2+\frac{2\gamma-\beta-1}{2\beta})$ with $\gamma=\gamma(\beta) \in (0,1)$, there exist constants $C_{38}=C_{38}(\beta,R)>0$, $C_{39}=C_{39}(\beta,R)>0$ and $C_{40}=C_{40}(\beta,R)>0$ such that the following three estimates hold
\begin{align}\label{lem-2.3.1.5}
  w(s,t) \leqslant C_{38} \psi^{\frac{1}{2\beta}}(t), 
  \quad (s,t) \in(0,R^2] \times (0,T_{\max})
\end{align}
and
\begin{align}\label{lem-2.3.1.1}
  \int_0^{R^2} \rho^{-l}w^2(\rho,t) \mathrm{d}{\rho}
  \leqslant C_{39}\psi^{\frac{1}{\beta}}(t), 
  \quad t \in  (0,T_{\max})
\end{align}
as well as
\begin{align}\label{lem-2.3.2.2}
  \phi(t) \leqslant C_{40} \psi^{\frac{1}{2\beta}}(t),
  \quad t \in\left(0, T_{\max }\right).
\end{align}
\end{lem}
\begin{proof}
Since $\beta>1$ and $\gamma>0$, we can choose $l=l(\beta)$ fulfilling 
\begin{align*}
  1<l<2+\frac{2\gamma-\beta-1}{2\beta},
\end{align*}
which implies that
\begin{align}\label{lem-2.3.1.6}
  \frac{2\gamma-\beta+1}{2\beta-2}
  > \frac{2\beta l-4\beta+2}{2\beta-2}
  = -1+\frac{\beta(l-1)}{\beta-1}>-1.
\end{align}
Applying Hölder's inequality, and using (\ref{2.3.0.2}) and $\beta>1$, we obtain
\begin{align}\label{lem-2.3.1.3}
  \frac{1}{2}w^2(\rho,t)
  = &\int_0^{\rho}w(s,t)w_s(s,t) \mathrm{d}{s} \notag \\
  \leqslant & \Big( \int_0^\rho s^{\frac{\beta-1}{2}-\gamma}w(s,t)w_s^{\beta}(s,t) \mathrm{d}s\Big)^\frac{1}{\beta} \cdot  
              \Big(\int_0^\rho s^{\frac{2\gamma-\beta+1}{2\beta-2}} w(s,t) \mathrm{d}s \Big)^{\frac{\beta-1}{\beta}} \notag\\
  \leqslant & \psi^{\frac{1}{\beta}}(t) \Big(\int_0^{\rho} s^{\frac{2\gamma-\beta+1}{2\beta-2}} w(s,t) \mathrm{d}s \Big)
              ^{\frac{\beta-1}{\beta}}.
\end{align}
Due to (\ref{lem-2.3.1.6}) and $w \leqslant \frac{{m_0}}{2\pi}$, we have 
\begin{align}\label{lem-2.3.1.10}
  \Big(\int_0^{\rho} s^{\frac{2\gamma-\beta+1}{2\beta-2}} w(s,t) \mathrm{d}s \Big)^{\frac{\beta-1}{\beta}}
  \leqslant & \Big(\frac{{m_0}}{2 \pi}\Big)^{\frac{\beta-1}{\beta}}  
              \Big(\int_0^\rho s^{\frac{2\gamma-\beta+1}{2\beta-2}}\mathrm{d}s\Big)
              ^{\frac{\beta-1}{\beta}} \notag\\
  = &\Big(\frac{2\beta-2}{2\gamma+\beta-1} \cdot \frac{{m_0}}{2 \pi}\Big)^\frac{\beta-1}{\beta}  \rho^{\frac{2\gamma+\beta-1}{2\beta}}.
\end{align}
Inserting (\ref{lem-2.3.1.10}) into (\ref{lem-2.3.1.3}), we infer that
\begin{align}\label{lem-2.3.1.11}
   \frac{1}{2}w^2(\rho,t)= \int_0^{\rho}w(s,t)w_s(s,t) \mathrm{d}{s}
  \leqslant &  \Big(\frac{2\beta-2}{2\gamma+\beta-1} \cdot \frac{{m_0}}{2 \pi}\Big)^\frac{\beta-1}{\beta} 
          \psi^{\frac{1}{\beta}}(t)  
          \rho^{\frac{2\gamma+\beta-1}{2\beta}}.
\end{align}
Due to $\rho \leqslant R^2$, there exists a constant $C_{38}=C_{38}(\beta,R)=\sqrt{2}(\frac{2\beta-2}{2\gamma+\beta-1} \cdot \frac{{m_0}}{2 \pi})^\frac{\beta-1}{2\beta}R^{\frac{2\gamma+\beta-1}{2\beta}}>0$ such that 
\begin{align}\label{lem-2.3.1.5-1}
  w(\rho,t) \leqslant C_{38} \psi^{\frac{1}{2\beta}}(t), 
  \quad (\rho,t) \in(0,R^2] \times (0,T_{\max}),
\end{align}
which is (\ref{lem-2.3.1.5}). Combining (\ref{lem-2.3.1.11}) with the fact that $l<2+\frac{2\gamma-\beta-1}{2\beta}$, there exists a positive constant $C_{39}=C_{39}(\beta,R)=2(\frac{2\beta-2}{2\gamma+\beta-1} \cdot \frac{{m_0}}{2\pi})^\frac{\beta-1}{\beta} R^{2-2l+{\frac{2\gamma+\beta-1}{\beta}}} $ such that
\begin{align*}
  \int_0^{R^2} \rho^{-l}w^2(\rho,t) \mathrm{d}{\rho}
  \leqslant & 2\Big(\frac{2\beta-2}{2\gamma+\beta-1} \cdot \frac{{m_0}}{2 \pi}\Big)^\frac{\beta-1}{\beta} 
         \psi^{\frac{1}{\beta}}(t)\cdot 
         \int_0^{R^2} \rho^{-l+{\frac{2\gamma+\beta-1}{2\beta}}} \mathrm{d}\rho    \notag\\
  \leqslant & C_{39} \psi^{\frac{1}{\beta}}(t), 
  \quad t \in (0,T_{\max}),
\end{align*}
which is (\ref{lem-2.3.1.1}). Due to $\gamma<1$ and (\ref{lem-2.3.1.5-1}), there exists a positive constant $C_{40}=C_{40}(\beta,R)=C_{38}\frac{R^{1-\gamma}}{1-\gamma}$ such that
\begin{align*}
  \phi(t)=\int_0^{R^2} s^{-\gamma} w(s, t) \mathrm{d} s 
  \leqslant C_{40} \psi^{\frac{1}{2\beta}}(t),
  \quad t \in (0,T_{\max}),
\end{align*}
which is (\ref{lem-2.3.2.2}).
\end{proof}
In the following lemma, we establish a basic differential inequality of the moment-type functional $\phi(t)$ by making use of (\ref{lem-2.3.1.1}) and (\ref{lem-2.3.1.5}).
\begin{lem}\label{lem-2.3.2}
Assume that $S(\xi)$ satisfies $(\ref{1.0.3})$ and $(\ref{1.0.4})$ with $\beta>1$ and $k>0$. Then for any $u_0$ satisfying $(\ref{1.0.2})$ and some $\gamma=\gamma(\beta) \in (0,1)$, there exists a constant $C_{41}=C_{41}(\beta,R)>0$, such that the functionals $\phi(t)$ and $\psi(t)$ satisfy
\begin{align}\label{lem-2.3.2.1}
  \phi^{\prime}(t) 
  \geqslant \frac{M}{C_{41}} \cdot \frac{\psi(t)}{1+\psi^{\frac{1}{2\beta}}(t)}
       -C_{41} \psi^{\frac{1}{2\beta}}(t) -C_{41},
  \quad t\in (0,T_{\max}).
\end{align}
\end{lem}
\begin{proof}
For any $\beta>1$, we fix $\gamma=\gamma(\beta) \in (0,1)$ fulfilling 
\begin{align*} 
  0<\gamma<\frac{\beta-1}{4\beta-2},
\end{align*}
which implies 
\begin{align}\label{lem-2.3.2.4}
  \left(2\gamma+1\right)-\Big(2+\frac{2\gamma-\beta-1}{2\beta}\Big)
  =\frac{\gamma(4\beta-2)-\beta+1}{2\beta}<0.
\end{align}
Thus, we choose $l=l(\beta)>0$ such that
\begin{align}\label{lem-2.3.2.5}
  2\gamma+1< l<2+\frac{2\gamma-\beta-1}{2\beta}.
\end{align}
According to (\ref{lem2.2.3.1}) and the definition of $\phi$, it is obvious that
\begin{align}\label{lem-2.3.2.3}
  \phi^{\prime}(t) 
  \geqslant 4 {\int_0^{R^{2}}}s^{1-\gamma}  w_{s s}(s, t)\mathrm{d}s
       +M kC_{37}  \int_0^{R^2}\frac{s^{\frac{\beta-1}{2}-\gamma}w(s, t) w_s^\beta(s, t)}{1+\frac{1}{2}
       \int_0^s \frac{w(\rho, t)}{\rho} \mathrm{d} \rho} \mathrm{d}s
\end{align}
for all $ t \in\left(0, T_{\max }\right)$. Due to $\gamma<1$, we obtain
\begin{align}\label{lem-2.3.2.12}
  s^{1-\gamma}w_s(s,t)\rightarrow 0 \text{ and } s^{-\gamma}w(s,t)\rightarrow 0 \ \text{as}\
  s\rightarrow 0,  \quad  t \in (0,T_{\max}).
\end{align}
Integrating by parts, and combining (\ref{lem-2.3.2.12}) with $w_s(s,t)\geqslant 0$ inferred from (\ref{lem-2.1.2.2}), we deduce that
\begin{align}\label{lem-2.3.2.7}
  4 {\int_0^{R^{2}}}s^{1-\gamma}  w_{s s}(s, t)\mathrm{d}s 
  = &-4(1-\gamma) \int_0^{R^{2}} s^{-\gamma}w_s(s,t) \mathrm{d}s 
      +4R^{2 (1-\gamma)} w_s(R^2,t) \notag\\
  \geqslant & -4(1-\gamma) \int_0^{R^{2}} s^{-\gamma}w_s(s,t) \mathrm{d}s \notag\\
  = &-4\gamma (1-\gamma) \int_0^{R^2} s^{-\gamma-1}w(s,t) \mathrm{d}s 
     -4(1-\gamma)\frac{R^{-2\gamma}{m_0}}{2 \pi}.
\end{align}
By means of Hölder's inequality, along with (\ref{lem-2.3.2.5}) and (\ref{lem-2.3.1.1}), we find that
\begin{align}\label{lem-2.3.2.8}
  \int_0^{R^2}s^{-\gamma-1}w(s,t) \mathrm{d}s 
   \leqslant & \Big(\int_0^{R^2}s^{-l}w^2(s,t) \mathrm{d}s\Big)^{\frac{1}{2}} \cdot 
          \Big(\int_0^{R^2}s^{l-2\gamma-2} \mathrm{d}s\Big)^{\frac{1}{2}}\notag\\
   \leqslant &  C_{39}^{\frac{1}{2}} \Big(\int_0^{R^2}s^{l-2\gamma-2} \mathrm{d}s\Big)^{\frac{1}{2}}  
           \psi^{\frac{1}{2\beta}}(t) \ \notag\\
   = &C_{42} \psi^{\frac{1}{2\beta}}(t), 
\end{align}
where $C_{42}=C_{42}(\beta,R)=C_{39}^{\frac{1}{2}}\Big(\int_0^{R^2}s^{l-2\gamma-2} \mathrm{d}s\Big)^{\frac{1}{2}}$. Inserting (\ref{lem-2.3.2.8}) into (\ref{lem-2.3.2.7}), we obtain
\begin{align}\label{lem-2.3.2.9}
  4{\int_0^{R^{2}}}s^{1-\gamma}  w_{s s}(s, t)\mathrm{d}s 
  \geqslant -4\gamma (1-\gamma)C_{42}  
       \psi^{\frac{1}{2\beta}}(t)-4(1-\gamma)\frac{R^{-2\gamma}{m_0}}{2 \pi}.
\end{align}
Using Hölder's inequality, and applying (\ref{lem-2.3.1.1}) and (\ref{lem-2.3.2.5}), we find that
\begin{align}\label{lem-2.3.2.10}
  \int_0^s \frac{w(\rho, t)}{\rho} \mathrm{d} \rho 
  \leqslant & \Big(\int_0^{R^2} \rho^{-l}w^2(\rho,t)\mathrm{d}\rho\Big)^{\frac{1}{2}} \cdot 
         \Big(\int_0^{R^2} \rho^{l-2} \mathrm{d}\rho\Big)^ {\frac{1}{2}} \notag\\
  \leqslant & {C_{39}}^{\frac{1}{2}}{C_{43}}\psi^{\frac{1}{2\beta}}(t), 
\end{align}
where $C_{43}=C_{43}(R)= \left(\int_0^{R^2} \rho^{l-2} \mathrm{d}\rho\right)^{\frac{1}{2}}$.
Combining (\ref{lem-2.3.2.10}) with the definition of $\psi$ in (\ref{2.3.0.2}), we have
\begin{align}\label{lem-2.3.2.11}
  \int_0^{R^2}\frac{s^{\frac{\beta-1}{2}-\gamma}w(s, t) w_s^\beta(s, t)}{1+\frac{1}{2}\int_0^s \frac{w(\rho, t)}{\rho} \mathrm{d} \rho} \mathrm{d}s 
  \geqslant & \frac{2\psi(t)}{2+{C_{39}}^{\frac{1}{2}}{C_{43}}\psi^{\frac{1}{2\beta}}(t)} \notag \\
  \geqslant & \frac{2}{\max\left\{2,{C_{39}}^{\frac{1}{2}}{C_{43}}\right\}} \cdot \frac{\psi(t)}{1+\psi^{\frac{1}{2\beta}}(t)}.
\end{align}
Therefore, we choose $C_{41}=C_{41}(\beta,R)>0$ to be sufficiently large and infer (\ref{lem-2.3.2.1}) from (\ref{lem-2.3.2.3}), (\ref{lem-2.3.2.9}) and (\ref{lem-2.3.2.11}). 
\end{proof}
\emph{\textbf{Proof of Theorem \ref{0.1}.}}
We denote
\begin{align}\label{lem-2.4.1.2}
  \phi_0=\phi_0\left(u_0\right):=\int_0^{R^2} s^{-\gamma} w_0(s) \mathrm{d} s
\end{align}
and
\begin{align}\label{lem-2.4.1.3}
  S:=\left\{t \in\left(0, T_{\max }\right) \mid \phi(t)>\frac{\phi_0}{2} \text { on }(0, t)\right\},
\end{align}
where $w_0(s)$ is defined in (\ref{lem-2.1.2.3}). Due to the continuity of $\phi(t)$, (\ref{2.3.0.1}) and (\ref{lem-2.4.1.2}), we infer that $S$ is non-empty. Thus, $T:=\sup S \in\left(0, T_{\max }\right] \subset(0, \infty]$ is well-defined. 

\textit{Step 1}: to prove that $T=T_{\max}$. Suppose that $T<T_{\max}$. According to the continuity of $\phi(t)$ and the definition of $T$, we have $\phi(T)=\frac{\phi_0}{2}$. In the following, we use (\ref{lem-2.3.2.1}) to prove $\phi(T) \geqslant \phi_0$, which gives a contradiction. 
 
 We first prove $\phi^{\prime}(t) \geqslant 0$ for all $t \in(0,T)$. As a consequence of (\ref{lem-2.3.2.2}) and (\ref{lem-2.4.1.3}), we obtain 
\begin{align}\label{lem-2.4.1.5}
  \psi(t) 
  \geqslant \Big(\frac{\phi(t)}{C_{40}}\Big)^{2\beta} 
  \geqslant \Big(\frac{\phi_0}{2 C_{40}}\Big)^{2\beta}, 
  \quad t \in(0,T).
\end{align}
Set
\begin{align*}
  f_M(z):=\frac{M}{2 C_{41}} \cdot \frac{z}{1+z^{\frac{1}{2\beta}}}-C_{41} z^{\frac{1}{2\beta}}-C_{41},
\end{align*}
then we have
\begin{align*}
  \inf _{z \geqslant C_{44}} f_M(z) \rightarrow+ \infty \text { \ as } \ M \rightarrow \infty,
\end{align*}
where $C_{44}:=\left(\frac{\phi_0}{2 C_{40}}\right)^{2\beta}$. According to (\ref{lem-2.4.1.5}), there exists a constant $M^{\star}(u_0)>0$ such that
\begin{align}\label{f_M}
  f_M(\psi(t)) \geqslant 0, \quad  t \in(0, T), \ M \geqslant M^{\star}(u_0).
\end{align}
We choose $ M \geqslant M^{\star}(u_0)$ in the following proof. Since $g(z):=\frac{z^{1-\frac{1}{2\beta}}}{1+z^{\frac{1}{2\beta}}}$ is nondecreasing on $z\in (0,+\infty)$, and from (\ref{lem-2.3.2.1}), (\ref{lem-2.3.2.2}), (\ref{lem-2.4.1.5}) and (\ref{f_M}), we deduce that
\begin{align}\label{lem-2.4.1.6}
  \phi^{\prime}(t) 
  \geqslant & \frac{M}{2 C_{41}} \cdot \frac{\psi(t)}{1+\psi^{\frac{1}{2\beta}}(t)}+f_M(\psi(t)) \notag\\
  \geqslant & \frac{M}{2 C_{41}} \cdot \frac{\psi^{1-\frac{1}{2\beta}}(t)}{1+\psi^{\frac{1}{2\beta}}(t)} {\psi}^{\frac{1}{2\beta}} \notag\\
  \geqslant & \frac{M}{2 C_{41}} \cdot \frac{\left(\frac{\phi_0}{2 C_{40}}\right)^{1-\frac{1}{2\beta}}}{1+(\frac{\phi_0}{2 C_{40}})^{\frac{1}{2\beta}}} \cdot \frac{\phi(t)}{C_{40}} \notag\\
  = &C_{45} \phi(t), 
  \quad  t \in(0, T),
\end{align} 
where $C_{45}=C_{45}\left(\beta,u_0, M\right)=\frac{M}{2 C_{40}C_{41}} \cdot \frac{\left(\frac{\phi_0}{2 C_{40}}\right)^{1-\frac{1}{2\beta}}}{1+(\frac{\phi_0}{2 C_{40}})^{\frac{1}{2\beta}}} $. By (\ref{lem-2.4.1.6}), we have $\phi^{\prime}(t) \geqslant 0$ for all $t \in(0,T)$, which, along with the continuity of $\phi(t)$, implies $\phi(T)\geqslant \phi_0$. 

\textit{Step 2}: to prove that $T_{\max}<\infty$. It follows from (\ref{lem-2.4.1.6}) that
\begin{align*}
  \phi(t) \geqslant \phi_0 e^{C_{45}t}, \quad t \in(0,T).
\end{align*}
Combining this with (\ref{2.3.0.1}), and using $w\leqslant \frac{{m_0}}{2\pi}$ and $\gamma<1$, we infer that
\begin{align*}
  \frac{{m_0} R^{2(1-\gamma)}}{2 \pi(1-\gamma)} 
  \geqslant \phi(t) 
  \geqslant \phi_0 e^{C_{45} t}, \quad t \in\left(0, T_{\max }\right),
\end{align*}
which implies 
\begin{align*}
  T_{\max } \leqslant \frac{1}{C_{45}} \ln \frac{{m_0} R^{2(1-\gamma)}}{2 \pi(1-\gamma) \phi_0}.
\end{align*}
In consequence, for any given $u_0$, if $ M \geqslant M^{\star}(u_0)$, we infer that $T_{\max } $ must be finite.
\hfill$\Box$


\begin{thebibliography}{10}

\bibitem{2023-CVPDE-AhnWinkler}
{\sc J.~Ahn and M.~Winkler}, {\em A critical exponent for blow-up in a
  two-dimensional chemotaxis-consumption system}, Calc. Var. Partial
   Differential Equations, 62 (2023),~Paper No. 180, pp. 25.
   
\bibitem{2024-AhnKangKim}
{\sc J.~Ahn, K.~Kang and D.~Kim}, {\em Global boundedness and blow-up in a repulsive
chemotaxis-consumption system in higher dimensions}, https://arxiv.org/abs/2408.16225, (2024).
   
\bibitem{2024-AML-DongZhangZhang}
{\sc Y.~Dong, S.~Zhang, and Y.~Zhang}, {\em Blowup phenomenon for a 2{D}
  chemotaxis-consumption model with rotation and signal saturation on the
  boundary}, Appl. Math. Lett., 149 (2024),~Paper No. 108934, pp. 6.

\bibitem{1971-JOTB-KELLERSEGEL}
{\sc E.F.~Keller and L.A.~Segel}, {\em Traveling bands of chemotactic bacteria: a
  theoretical analysis}, J Theor Biol, 30 (1971), pp.~235-248.
  
\bibitem{1970-TMMO-EidelmanIvasishen}
{\sc S.~D. {\`E}idel'man and S.~D. Ivasishen}, {\em Investigation of the
  green's matrix of a homogeneous parabolic boundary value problem}, Trudy
  Moskovskogo Matematicheskogo Obshchestva, 23 (1970), pp.~179--234.

\bibitem{2017-JMP-FanJin}
{\sc L.~Fan and H.-Y. Jin}, {\em Global existence and asymptotic behavior to a
  chemotaxis system with consumption of chemoattractant in higher dimensions},
  J. Math. Phys., 58 (2017),~Paper No. 011503, pp. 22.

\bibitem{2017-DCDS-LankeitWang}
{\sc J.~Lankeit and Y.~Wang}, {\em Global existence, boundedness and
  stabilization in a high-dimensional chemotaxis system with consumption},
  Discrete Contin. Dyn. Syst., 37 (2017), pp.~6099--6121.

\bibitem{2022-N-LankeitWinkler}
{\sc J.~Lankeit and M.~Winkler}, {\em Radial solutions to a
  chemotaxis-consumption model involving prescribed signal concentrations on
  the boundary}, Nonlinearity, 35 (2022), pp.~719--749.

\bibitem{2023-SAM-LankeitWinkler}
\leavevmode\vrule height 2pt depth -1.6pt width 23pt, {\em Depleting the
  signal: analysis of chemotaxis-consumption models---a survey}, Stud. Appl.
  Math., 151 (2023), pp.~1197--1229.

\bibitem{2015-MMMAS-LiSuenWinklerXue}
{\sc T.~Li, A.~Suen, M.~Winkler, and C.~Xue}, {\em Global small-data solutions
  of a two-dimensional chemotaxis system with rotational flux terms}, Math.
  Models Methods Appl. Sci., 25 (2015), pp.~721--746.

\bibitem{2015-BVP-Li}
{\sc X.~Li}, {\em Global existence and uniform boundedness of smooth solutions
  to a parabolic-parabolic chemotaxis system with nonlinear diffusion}, Bound.
  Value Probl.,  (2015), ~2015:107, 17 pp.

\bibitem{1990-PA-MATSUSHITAFUJIKAWA}
{\sc M.~Matsushita and H.~Fujikawa}, {\em Diffusion-limited growth in bacterial
  colony formation}, Physica A, 168 (1990), pp.~498--506.
  
\bibitem{2019-QuittnerSouplet}
{\sc P.~Quittner and P.~Souplet}, {\em Superlinear parabolic problems},
  Springer, 2019.

\bibitem{2011-JMAA-Tao}
{\sc Y.~Tao}, {\em Boundedness in a chemotaxis model with oxygen consumption by
  bacteria}, J. Math. Anal. Appl., 381 (2011), pp.~521--529.

\bibitem{2012-JDE-TaoWinkler}
{\sc Y.~Tao and M.~Winkler}, {\em Boundedness in a quasilinear
  parabolic-parabolic {K}eller-{S}egel system with subcritical sensitivity}, J.
  Differential Equations, 252 (2012), pp.~692--715.

\bibitem{2012-JDE-TaoWinklera}
\leavevmode\vrule height 2pt depth -1.6pt width 23pt, {\em Eventual smoothness
  and stabilization of large-data solutions in a three-dimensional chemotaxis
  system with consumption of chemoattractant}, J. Differential Equations, 252
  (2012), pp.~2520--2543.

\bibitem{2014-JDE-TaoWinkler}
\leavevmode\vrule height 2pt depth -1.6pt width 23pt, {\em Energy-type
  estimates and global solvability in a two-dimensional chemotaxis-haptotaxis
  model with remodeling of non-diffusible attractant}, J. Differential
  Equations, 257 (2014), pp.~784--815.

\bibitem{2019-JDE-TaoWinkler}
\leavevmode\vrule height 2pt depth -1.6pt width 23pt, {\em Global smooth
  solvability of a parabolic-elliptic nutrient taxis system in domains of
  arbitrary dimension}, J. Differential Equations, 267 (2019), pp.~388--406.
  

\bibitem{2005-PTNASTUSA-TuvalCisnerosDombrowskiWolgemuthKesslerGoldstein}
{\sc I.~Tuval, L.~Cisneros, C.~Dombrowski, C.~Wolgemuth, J.~Kessler, and
  R.~Goldstein}, {\em Bacterial swimming and oxygen transport near contact
  lines}, Proc. Natl. Acad. Sci. USA, 102 (2005), pp.~2277--2282.
  
\bibitem{2019-JFA-Winkler}
{\sc M.~Winkler}, {\em A three-dimensional {K}eller-{S}egel-{N}avier-{S}tokes
  system with logistic source: global weak solutions and asymptotic
  stabilization}, J. Funct. Anal., 276 (2019), pp.~1339--1401.
  
\bibitem{2018-N-Winkler}
{\sc M.~Winkler}, {\em A critical blow-up exponent in a chemotaxis system with
  nonlinear signal production}, Nonlinearity, 31 (2018), pp.~2031--2056.

\bibitem{2019-EJDE-WangLi}
{\sc H.~Wang and Y.~Li}, {\em Renormalized solutions to a chemotaxis system
  with consumption of chemoattractant}, Electron. J. Differential Equations,
  (2019),~Paper No. 38, pp. 19.

\bibitem{2013-EJDE-WangKhanKhan}
{\sc L.~Wang, S.~U.-D. Khan, and S.~U.-D. Khan}, {\em Boundedness in a
  chemotaxis system with consumption of chemoattractant and logistic source},
  Electron. J. Differential Equations,  (2013),~paper No. 209, pp. 9.

\bibitem{2016-ZAMP-WangMuHu}
{\sc L.~Wang, C.~Mu, and X.~Hu}, {\em Global solutions to a chemotaxis model
  with consumption of chemoattractant}, Z. Angew. Math. Phys., 67 (2016),
  ~Art. 96, pp. 16.

\bibitem{2018-AA-WangMuHuZheng}
{\sc L.~Wang, C.~Mu, X.~Hu, and P.~Zheng}, {\em Boundedness in a quasilinear
  chemotaxis model with consumption of chemoattractant and logistic source},
  Appl. Anal., 97 (2018), pp.~756--774.

\bibitem{2015-ZAMP-WangMuLinZhao}
{\sc L.~Wang, C.~Mu, K.~Lin, and J.~Zhao}, {\em Global existence to a
  higher-dimensional quasilinear chemotaxis system with consumption of
  chemoattractant}, Z. Angew. Math. Phys., 66 (2015), pp.~1633--1648.

\bibitem{2014-ZAMP-WangMuZhou}
{\sc L.~Wang, C.~Mu, and S.~Zhou}, {\em Boundedness in a parabolic-parabolic
  chemotaxis system with nonlinear diffusion}, Z. Angew. Math. Phys., 65
  (2014), pp.~1137--1152.

\bibitem{2023-PRSESA-WangWinkler}
{\sc Y.~Wang and M.~Winkler}, {\em Finite-time blow-up in a repulsive
  chemotaxis-consumption system}, Proc. Roy. Soc. Edinburgh Sect. A, 153
  (2023), pp.~1150--1166.

\bibitem{2015-ZAMP-WangXiang}
{\sc Y.~Wang and Z.~Xiang}, {\em Global existence and boundedness in a
  higher-dimensional quasilinear chemotaxis system}, Z. Angew. Math. Phys., 66
  (2015), pp.~3159--3179.
  
\bibitem{2024-NARWA-YangAhn}
{\sc S.-O. Yang and J.~Ahn}, {\em Long time asymptotics of small mass solutions
  for a chemotaxis-consumption system involving prescribed signal
  concentrations on the boundary}, Nonlinear Anal. Real World Appl., 79 (2024),
  ~Paper No. 104129, 16 pp.

\bibitem{2015-JMP-ZhangLi}
{\sc Q.~Zhang and Y.~Li}, {\em Stabilization and convergence rate in a
  chemotaxis system with consumption of chemoattractant}, J. Math. Phys., 56
  (2015),~Paper No. 081506, pp. 10.

\end{thebibliography}
\end{document}